\documentclass[12pt,a4paper]{amsart}
\usepackage[utf8]{inputenc}
\usepackage[english]{babel}
\usepackage[T1]{fontenc}
\usepackage{amsmath,interval}
\usepackage{vmargin}
\usepackage{mathtools}
\usepackage{amsfonts}
\usepackage{amssymb}
\usepackage{setspace}
\usepackage{accents}
\usepackage{mathrsfs, enumerate, csquotes, color}
\usepackage{xcolor}
\definecolor{vertfonce}{rgb}{0.20, 0.46, 0.25}
\definecolor{rougefonce}{rgb}{0.64, 0.09, 0.20}
\usepackage[breaklinks=true,
            colorlinks=true,
            linkcolor=rougefonce,
            citecolor=vertfonce]{hyperref}

\author{L. Morin}
\email{leo.morin@ens-rennes.fr}
\address{Aarhus Univ., Ny Munkegade 118, DK-8000 Aarhus C, Denmark}

\author{N. Raymond}
\email{nicolas.raymond@univ-angers.fr}
\address{Univ Angers, CNRS, LAREMA, SFR MATHSTIC, F-49000 Angers, France}

\author{S. V\~u Ng\d{o}c}
\email{san.vu-ngoc@univ-rennes1.fr}
\address{Univ Rennes, CNRS, IRMAR - UMR 6625, F-35000 Rennes, France}
\title[]{Eigenvalue asymptotics for confining magnetic Schrödinger operators with complex potentials}

\makeatletter

\@addtoreset{equation}{section}
\makeatother

\usepackage{tikz}
\usetikzlibrary{shapes.misc}
\tikzset{cross/.style={cross out}, minimum size=1pt, draw=black, inner sep =0pt, outer sep=0pt, cross/.default={1pt}}

\theoremstyle{plain}
\newtheorem{theorem}{Theorem}[section]
\newtheorem{lemma}[theorem]{Lemma}
\newtheorem{corollary}[theorem]{Corollary}
\newtheorem{proposition}[theorem]{Proposition}

\theoremstyle{definition}
\newtheorem{remark}[theorem]{Remark}
\newtheorem{notation}{Notation}

\newtheorem{definition}[theorem]{Definition}
\newtheorem{assumption}{Assumption}

\newcommand{\A}{\mathbf{A}}
\newcommand{\R}{\mathbf{R}}
\newcommand{\C}{\mathbf{C}}
\newcommand{\N}{\mathbf{N}}

\newcommand{\spectrum}{\mathsf{sp}}

\newcommand{\supp}{\mathsf{supp}}

\renewcommand{\Re}{\mathrm{Re}\,}
\renewcommand{\Im}{\mathrm{Im}\,}
\renewcommand{\leq}{\leqslant}	\renewcommand{\geq}{\geqslant}

\newcommand{\Lh}{\mathscr{L}_h}
\newcommand{\Loh}{\widehat{\mathscr{L}}^0_h}

\newcommand{\Peff}{\mathsf{P}^{\mathsf{eff}}_h}

\newcommand{\Op}{\mathsf{Op}}
\newcommand{\Oph}{\mathsf{Op}_h}
\newcommand{\Ld}{\mathsf{L}^{2}}
\newcommand{\Cinf}{\mathcal{C}^{\infty}}
\newcommand{\grandO}{\mathscr{O}}
\newcommand{\dd}{\mathrm{d}}

\newcommand{\abs}[1]{\left|#1\right|}
\newcommand{\norm}[1]{\left\|#1\right\|}

\newcommand{\rond}[1]{\accentset{\circ}{#1}}

\begin{document}

\maketitle

\begin{abstract}
This article is devoted to the spectral analysis of the electro-magnetic Schrödinger operator on the Euclidean plane. In the semiclassical limit, we derive a pseudo-differential effective operator that allows us to describe the spectrum in various situations and appropriate regions of the complex plane. Not only results of the selfadjoint case are proved (or recovered) in the proposed unifying framework, but new results are established when the electric potential is complex-valued. In such situations, when the non-selfadjointness comes with its specific issues (lack of a "spectral theorem", resolvent estimates), the analogue of the "low-lying eigenvalues" of the selfadjoint case are still accurately described and the spectral gaps estimated.
\end{abstract}

\section{Introduction}

\subsection{Context and motivation}

In this article we study the spectrum of the non-selfadjoint
electromagnetic Schrödinger operator:
\begin{equation}\label{eq.mainLh}
\mathscr{L}_{h} =(-ih\nabla-\mathbf{A})^2 + h V(q_1,q_2)\,,
\end{equation}
which is an unbounded differential operator on $L^2(\R^2)$. We are particularly concerned by the
semiclassical limit $h\to 0$. Here $\mathbf{A} : \R^2\to\R^2$ is a
smooth vector potential and $V : \R^2 \rightarrow \C$ a smooth \emph{complex}
scalar potential. The associated magnetic field
$B=\partial_1A_2-\partial_2A_1$ is assumed to be positive and to
belong to the class of \emph{bounded symbols}
\[S_{\R^2}(1)=\{f\in\mathscr{C}^\infty(\R^2;\C) : \forall\alpha\in\N\,,\exists C>0\,, |\partial^\alpha f|\leq C\}\,.\]
We also assume that the complex perturbation $V$ belongs to this class, \emph{i.e.}, $V \in S_{\R^2}(1)$, see Assumptions \ref{Assumption.B.positif} and \ref{Assumption.classeS}. The chosen order of magnitude of the electric interaction $hV$ is precisely when the magnetic and electric fields are in competition, as we will see in our results.

When $V=0$, the low-lying spectrum of this operator has been studied in several papers, and summarized in the books \cite{FournaisHelffer, Raymond}. In particular, when the magnetic field has a unique minimum $b_0 >0$, which is non-degenerate and not attained at infinity, it was proved in \cite{HelKor14-2} that
\[\lambda_n(h)=b_0 h+((2n-1)c_0+c_1)h^2+o(h^2)\,,\]
where $c_1\in\R$ and $c_0= b_0^{-1} \vert \frac{1}{2}\nabla^2 B(q_0) \vert^{1/2}$.

Such operators as \eqref{eq.mainLh} appear for instance in the context of the
time-dependent Ginzburg-Landau equations, see~\cite{ALP-R2}. These
equations involve a damping term related to an induced current. Their
linearization near a normal state gives rise to a propagation equation
whose generator is an electromagnetic Schrödinger operator with
complex electric potential. Its left-most eigenvalue governs the large
time decay of the associated semi-group and thus the stability of the
normal state.

Our analysis sets up a unifying (semiclassical) framework to study the "low-lying"
eigenvalues of operators of the form $\mathscr{L}_{h}$, including the
magnetic Laplacian itself \cite{HelKor14-2,Birkhoff2D} and some of its
selfadjoint perturbations (see the recent work \cite{Yoshida}). But the most
interesting novelty of our strategy is to cover also the case of \emph{imaginary}
electric potentials. More precisely, to the authors’ knowledge, the present paper is the first
to obtain precise eigenvalue asymptotics in the presence of strong\footnote{The perturbation is not assumed to be small at infinity and plays at the same scale as the magnetic Laplacian.} perturbations of the magnetic Schrödinger operator in the semiclassical limit. Note that, in non-asymptotic settings, complex perturbations of the magnetic Laplacian (and of the magnetic Dirac operator) have been considered in \cite{FKV18, CFK20}, where it is proved that there are no eigenvalues when the electro-magnetic field is sufficiently decaying at infinity. For decaying electric potentials with a (quasi)constant magnetic field, Weyl estimates have also been established for the Pauli operator in \cite{Sambou2,Sambou}.

In general, it is well-known that even small non-selfadjoint perturbations of selfadjoint operators can have a dramatic effect on
the spectrum. In the present context, this problem is all the more appealing that the magnetic Laplacian comes with its own issues (such as its lack of ellipticity). To overcome these combined difficulties, our approach is based on a microlocal dimensional
reduction (involving operator-valued symbols, see for instance
\cite{Martinez07,Keraval}, and also \cite{BHR21} for a recent application of the strategy in a self-adjoint context). It allows us to explore the spectral
structure in a disk corresponding to the location of the low-lying
eigenvalues in the selfadjoint case. More precisely, when the
perturbation $V$ is turned on, we describe how the
spectrum moves, in a disc $D(\mu_0 h, C h^2)$, where $\mu_0$ depends on the electromagnetic field. In this disc, the fine structure of the spectrum is accurately described by estimating the splitting between the eigenvalues. Our main result is stated in Theorem \ref{Main-Thm}, whereas its various (and sometimes non-trivial) applications are given in Section \ref{sec.corollaries}.

\subsection{Main result}

In this article we will make the following assumptions. 

\begin{assumption}\label{Assumption.B.positif}
  The magnetic field is non-vanishing: there exists $b_0 >0$ such
  that $$\forall q \in \R^2, \quad B(q) \geq b_0 >0.$$
\end{assumption}

\begin{assumption}
  \label{Assumption.classeS}
  The functions $B$, $V$ and
  \[
    q \mapsto \int_0^{q_1} \frac{\partial B(s,q_2)} {\partial {q_2}}
    \dd s
  \]
  are all in $S_{\R^2}(1)$.
\end{assumption}

\begin{assumption}\label{Assumption}
  There exist $u,v \in \R$, $u > 0$ such that the function
  \[ F = u (B+\mathsf{Re}(V)) + v \mathsf{Im}(V) =
    \mathsf{Re}((u-iv)(B+V))\] admits a unique global minimum, not
  reached at infinity. We denote by $\mu_0 \in \C$ the value of $B+V$
  at the minimum of $F$. It satisfies
  \[
    \mathsf{Re}((u-iv)\mu_0) = u \mathsf{Re}\mu_0 + v \mathsf{Im}
    \mu_0 = \min_{q \in \R^2} F(q).
  \]
\end{assumption}

The function $F$ should be interpreted as the localizing function for
our operator: it gives information on where the spectrum should
lie. These assumptions imply discreteness of the spectrum in a disc
$D(\mu_0 h, C h^2)$ and localization of the associated eigenfunctions
(Proposition \ref{prop.Lh.Fredholm} and Lemma
\ref{lem.microlocalisation}). Here are some interesting particular
cases where Assumption \ref{Assumption} holds:
\begin{enumerate}
\item[1.] $B + \mathsf{Re}(V)$ admits a unique global minimum, not reached at infinity, and $\mathsf{Im}(V)$ is arbitrary (take $u=1$ and $v=0$).\\
  
\item[2.] $\mathsf{Im}(V)$ admits a unique global minimum, not reached at infinity, and $B + \mathsf{Re}(V)$ is constant (take $u=1$ and $v=1$).
\end{enumerate}

\begin{remark}
  It may happen that there exist two different couples $(u,v)$ for
  which Assumption~\ref{Assumption} holds, and for which the
  corresponding minima of $F$ are attained at two different locations
  $q_0$ and $q'_0$. Then it follows from the assumptions that the
  values $\mu_0$ and $\mu_0'$ must be different as well. Hence our
  analysis will give the description of the spectrum of
  $\mathscr{L}_h$ in two different regions in the complex plane.
Here is an example of such a situation (see Figure \ref{fig.0}). Assume that the electric field is purely imaginary, equal to $iV(q)$ and that magnetic field is $B=1-w(q)$, where $w$ and $V$ have disjoint compact supports, $0 \leq w < 1$ and $V \geq 0$. If $w$ (resp. $V$) has a unique and non-degenerate maximum reached at $q_0$ (resp. reached at $q_0'$) then the functions $F= 1 - w(q)$ ($u=1$, $v=0$) and $F'= 1 - w(q) - V(q)$ ($u=1$, $v=-1$) satisfy our assumption with respective minima $1 - w(q_0)$ and $1 - \max( w(q_0), V(q_0'))$. If $V(q_0') > w(q_0)$ these minimas are reached at $q_0$ and $q_0'$ respectively and the corresponding values of $B+iV$ are 
\[ \mu_0 = 1-w(q_0) \,, \quad \mu_0' = 1+iV(q_0') \,.\]
Note that the value of $F$ at its minimum plays no role in the spectral description. The interesting quantity is $\mu_0$.

\begin{center}
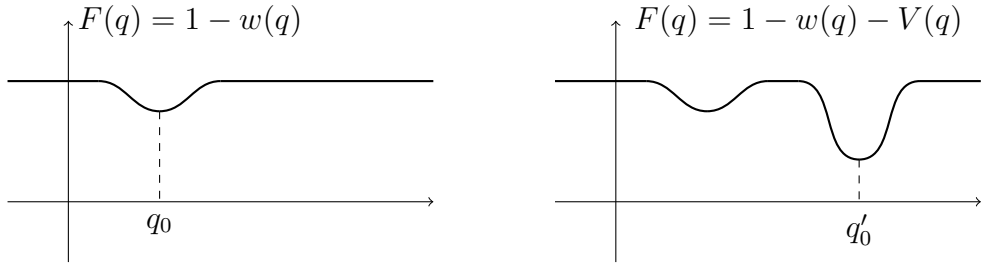
\begin{figure}[ht!]
\begin{tikzpicture}[scale=0.8]\label{fig:example1}
\node[] at (2,3) {$F(q)=1-w(q)$};
\draw[->] (-1,0)--(6,0);
\draw[->] (0,-1)--(0,3);
\draw[thick] (-1,2)--(0.5,2);
\draw[thick] (0.5,2) to[out=0, in=180] (1.5,1.5);
\draw[thick] (1.5,1.5) to [out=0, in=180] (2.5,2);
\draw[thick] (2.5,2)--(6,2);
\draw[dashed] (1.5,1.5)--(1.5,0);
\node[below] at (1.5,-0.2) {$q_0$};

\node[] at (12,3) {$F(q)=1-w(q)-V(q)$};
\draw[->] (8,0)--(15,0);
\draw[->] (9,-1)--(9,3);
\draw[thick] (8,2)--(9.5,2);
\draw[thick] (9.5,2) to[out=0, in=180] (10.5,1.5);
\draw[thick] (10.5,1.5) to [out=0, in=180] (11.5,2);
\draw[thick] (11.5,2)--(12,2);
\draw[thick] (12,2) to[out=0,in=180] (13,0.7);
\draw[thick] (13,0.7) to[out=0, in=180] (14,2);
\draw[thick] (14,2)--(15,2);
\draw[dashed] (13,0.7)--(13,0);
\node[below] at (13,-0.2) {$q_0'$};

\end{tikzpicture}
\caption{Two functions $F$}\label{fig.0}
\end{figure}
\end{center}

\end{remark}

In this article, we compare the spectrum of $\Lh$ in $D(\mu_0h, Ch^2)$
with the spectrum of an effective operator $h \Peff$ acting on
$L^2(\R)$, which has the following form.
Let us denote by $\varphi : \R^2 \rightarrow \R^2$ the  diffeomorphism:
\begin{equation}\label{eq.diffeo}
(\xi,x) = \varphi (q) = \left( \int_0^{q_1} B(s,q_2) \dd s \,, q_2 \right) \,,
\end{equation}
and $\rond{B}(\xi,x) = B \circ \varphi^{-1}(\xi,x)$,
$\rond{V}(\xi,x) = V \circ \varphi^{-1}(\xi,x)$. Then
 $\Peff = \Oph^w \mu^{\mathsf{eff}}_h$ is an $h$-pseudodifferential operator
with a symbol in $S_{\R^2}(1)$ of the form
\begin{equation}
\mu^{\mathsf{eff}}_h(x,\xi) = \rond{B}(\xi,x) + \rond{V}(\xi,x)+ h \mu_1(x,\xi),
\end{equation}
where the subprincipal term $\mu_1(x, \xi)$ has an explicit formula
described in \eqref{eq.mu-eff}. 

Our main result, Theorem \ref{Main-Thm}, states that, provided that
the eigenvalues and resolvent of $\Peff$ are sufficiently well
controled near $\mu_0$, the spectrum of $h \Peff$ in that region
closely approximate the spectrum of $\Lh$.

\begin{theorem}\label{Main-Thm}
  Let Assumptions \ref{Assumption.B.positif}, \ref{Assumption.classeS},
  and \ref{Assumption} hold. Let $0<C<C'$. Assume the following:
  \begin{enumerate}[\rm (a)]
  \item\label{thm.a} There exist $c,h_0>0$ such that, for $h \in (0,h_0)$, the
    spectrum of $\Peff$ in $D(\mu_0,C'h)$ consists of a family of
    discrete eigenvalues $(\nu_j(h))_{1 \leq j \leq N}$ of algebraic
    multiplicity 1, such that
\[\forall (k,\ell) \in \lbrace 1, \cdots , N \rbrace^2, \quad k \neq \ell \Rightarrow  \vert \nu_k(h) - \nu_{\ell}(h) \vert \geq c h\,.\]

\item There exist $\kappa\in(0,\frac{1}{2})$,
  $h_0$, and $C_0 >0$ such that, for $h \in (0,h_0)$, for any
  $z \in D(\mu_0,Ch)$ satisfying
  $\mathsf{dist}(z,\spectrum \ \Peff)\geq h^{\frac 3 2 - \kappa}$,
  \begin{equation}\label{eq.Res-bound}
    \Vert (z- \Peff)^{-1} \Vert \leq
    \frac{C_0}{\mathsf{dist}(z,\spectrum \ \Peff)}\,.
  \end{equation}

\end{enumerate}
Then, for $h$ small enough, the spectrum of $\Lh$ in $D(\mu_0 h,Ch^2)$ consists of a family of discrete eigenvalues $(\lambda_j(h))_{1 \leq j \leq N}$ of algebraic multiplicity 1, such that
$$\lambda_j(h) = h \nu_j(h) + \grandO(h^{\frac{5}{2}- \kappa}).$$ 
\end{theorem}

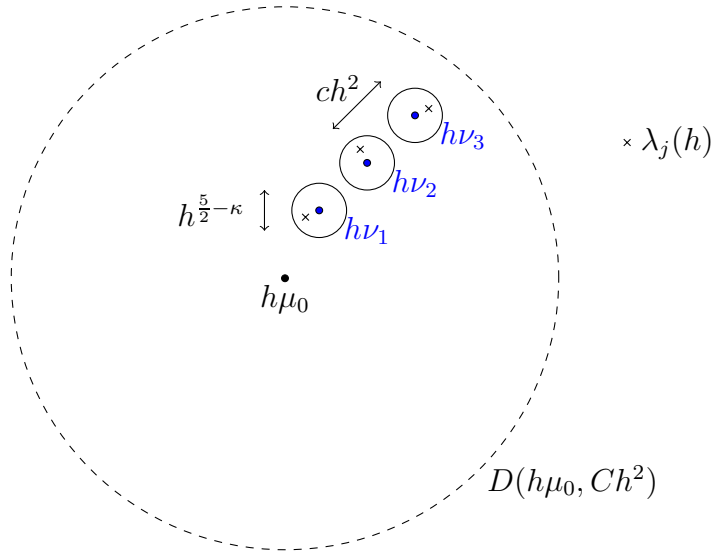
\begin{center}
\begin{tikzpicture}[scale=0.9]\label{fig:Landau-levels}
\draw[dashed] (0,0) circle (4);
\node[below] at (0,-0.1) {$h\mu_0$};
\draw[fill=black] (0,0) circle (0.05);

\node[below,color=blue] at (1.2,0.9) {$h \nu_1$};
\draw[fill=blue] (0.5,1) circle (0.05);
\draw[] (0.5,1) circle (0.4);

\node[below,color=blue] at (1.9,1.6) {$h \nu_2$};
\draw[fill=blue] (1.2,1.7) circle (0.05);
\draw[] (1.2,1.7) circle (0.4);

\node[below,color=blue] at (2.6,2.3) {$h \nu_3$};
\draw[fill=blue] (1.9,2.4) circle (0.05);
\draw[] (1.9,2.4) circle (0.4);

\draw[<->] (1.4,2.9)--(0.7,2.2);
\node[left] at (1.1,2.8) {$c h^2$};

\draw[<->] (-0.3,0.7)--(-0.3,1.3);
\node[left] at (-0.6,1) {$h^{\frac 5 2 - \kappa}$};

\draw plot[mark=x] coordinates{(0.3,0.9)};
\draw plot[mark=x] coordinates{(1.1,1.9)};
\draw plot[mark=x] coordinates{(2.1,2.5)};

\node[] at (4.2,-3) {$D(h\mu_0,Ch^2)$};

\draw plot[mark=x] coordinates{(5,2)};
\node[right] at (5.2,2) {$\lambda_j(h)$};
\end{tikzpicture}
\end{center}

\begin{remark}
  This theorem holds for $N=0$ as well, meaning that if $\Peff$ has no
  spectrum in $D(\mu_0,Ch)$ then $\Lh$ has no spectrum in
  $D(\mu_0 h , Ch^2)$.
\end{remark}

\subsection{Applications of the main theorem}\label{sec.corollaries}

Although the main incentive for Theorem \ref{Main-Thm} is to deal with
non-selfadjoint versions of the electromagnetic Schrödinger
operator, it turns out that Theorem \ref{Main-Thm} also recovers and
extends some recent results in the selfadjoint case.

\begin{corollary}[Self-adjoint case]
  Let Assumptions \ref{Assumption.B.positif}, \ref{Assumption.classeS}
  hold, and assume moreover that $\mathsf{Im}(V) = 0$, and
  $B + \mathsf{Re}(V)$ admits a unique minimum, which is
  non-degenerate, $\mu_0 \in \R$ reached at $0$ and not at
  infinity. Let $C >0$. There exists $h_0>0$ such that, for
  $h \in (0,h_0)$ the spectrum of $\Lh$ in $D(\mu_0 h, Ch^2)$ consists
  of a family of discrete eigenvalues
  $(\lambda_j(h))_{1 \leq j \leq N}$ of simple multiplicities such
  that
\[\lambda_j(h) = \mu_0 h + ((2j-1)c_0 + c_1)h^2 + o(h^2)\,,\]
where
$c_0 = B(0)^{-1} \vert \frac{1}{2}\nabla^2 ( B + V ) (0) \vert^{1/2}$
and $c_1 \in \R$.
\end{corollary}

\begin{proof}
  When $\mathsf{Im}(V) = 0$, $\Peff$ and $\Lh$ are (essentially)
  selfadjoint. In particular, the resolvent bound
  \eqref{eq.Res-bound} follows from the Spectral Theorem. The
  assumptions on $B + V$ imply that the spectrum of $\Peff$ in
  $D(\mu_0, Ch)$ consists of discrete simple eigenvalues such that
\[\nu_j(h) = \mu_0 + \left((2j-1) \vert \frac{1}{2}\nabla^2 (\hat{B} + \hat{V})(0)\vert^{1/2} + c_1\right) h + o(h)\,.\]
(See \cite{Sj92} for instance). The computation of $c_0$ follows from
the link between $\hat{B} + \hat{V}$ and $B + V$.
\end{proof}
In the case $V=0$, the above result was proven
in~\cite{Birkhoff2D,HelKor14-2}; the general case seems to be new.

We now turn to the non-selfadjoint case, where a remarkable
consequence of Theorem \ref{Main-Thm} is the following result.
\begin{theorem}\label{thm.nsa} Let  $p = B + V$;
  together with Assumptions \ref{Assumption.B.positif},
  \ref{Assumption.classeS}, and \ref{Assumption}, assume that
  $p^{-1}(\mu_0) \cap \R^2 = \lbrace 0 \rbrace$, with
  $\nabla p (0) = 0$ and that $\nabla^2 p(0)$ is non-degenerate. For
  $C>0$, there exists $h_0 >0$ such that, for $h \in (0,h_0)$ the
  spectrum of $\Lh$ in $D(\mu_0h,Ch^2)$ consists of a family of
  discrete eigenvalues $(\lambda_j(h))_{1 \leq j \leq N}$ with simple
  algebraic multiplicities such that
  \[\lambda_j(h) = \mu_0 h + \left( (2j-1)c_0 + c_1 \right) h^2 +
    o(h^2)\,,\] where $c_0, c_1 \in \C$, $c_0 \neq 0$.
\end{theorem}

\begin{proof}
  To apply Theorem \ref{Main-Thm}, we need to check that $\Peff$ has
  the expected spectral properties in $D(\mu_0, Ch)$. They are
  established in Section \ref{sec.7.Hitrik}.
\end{proof}
\begin{remark}
  Operators like $\Peff$ have been studied in \cite{Hitrik04}, where a
  full asymptotic expansion was provided thanks to a Birkhoff normal
  form in a non-selfadjoint context. The method used there requires
  that the symbol be analytic in a tubular neighborhood of $\R^2$. Our
  strategy to reduce the spectral analysis to the one of $\Peff$ could
  actually give us an effective operator modulo
  $\mathscr{O}(h^\infty)$. Combined with Hitrik's result, one would
  get a full asymptotic expansion of the eigenvalues in $D(\mu_0,Ch)$.
\end{remark}

Theorem~\ref{thm.nsa} can also be applied to two other quite different
interesting situations, where the confinement is given either by the
imaginary part of $V$, or by the magnetic field $B$ alone.

\begin{corollary}[Constant real part]
  Let assumptions \ref{Assumption.B.positif}, \ref{Assumption.classeS}
  hold, and assume that $B + \mathsf{Re} (V)$ is constant on $\R^2$
  equal to $\mu_0$. Assume that $\mathsf{Im}(V)$ admits a unique
  minimum, which is non-degenerate, reached at $0$ and not at
  infinity, with $\mathsf{Im} V(0) = 0$. For $C>0$, there exists
  $h_0 >0$ such that, for $h \in (0,h_0)$ the spectrum of $\Lh$ in
  $D(\mu_0h, Ch^2)$ consists of a family of discrete eigenvalues
  $(\lambda_j(h))_{1 \leq j \leq N}$ with simple algebraic
  multiplicities such that
$$\lambda_j(h) =  \mu_0 h + ((2j-1) i c_0 + c_1)h^2 + o(h^2),$$
where $c_0 >0$ and $c_1 \in \C$.
\end{corollary}

\begin{proof}
  In this case, Assumption \ref{Assumption} is valid with $u=v=1$. We
  can apply Theorem~\ref{thm.nsa}. Indeed, we have
  $p = \mu_0 + i \mathsf{Im}(V)$ so that
  $p(\R^2) \subset \mu_0 + i \R_+$ and $0$ is the only point where
  $p = \mu_0$. The Hessian of $p$ is $i \nabla^2 \mathsf{Im} V (0)$,
  which is non degenerate.
\end{proof}

\medskip

Finally, we can also describe the spectrum of the following
non-selfadjoint perturbation of the magnetic Laplacian:
\[ \mathscr{L}_{h,\varepsilon} = (-ih \nabla - A)^2 + h \varepsilon V\,,\]
where the magnetic field $B$ admits a unique minimum $b_0 >0$ which is non-degenerate, reached at 0 and not at infinity. In this case, $B + \varepsilon V$ admits a unique critical point $z_{\varepsilon}$ such that $z_{\varepsilon} = \grandO(\varepsilon)$, and we denote the critical value by $\mu_{\varepsilon} = B(z_\varepsilon) + \varepsilon V(z_\varepsilon)$.

\begin{corollary}[Perturbation of a confining magnetic
  field]\label{coro.perturbation}
  Let Assumptions \ref{Assumption.B.positif}, \ref{Assumption.classeS}
  hold, and assume that $B$ has a unique minimum, which is
  non-degenerate, $b_0$, reached at $0$ and not at infinity. Let
  $C >0$. There exist $\varepsilon_0, h_0 >0$ such that for
  $\varepsilon \in (0, \varepsilon_0)$ and $h \in (0, h_0)$
  the spectrum of $\mathscr{L}_{h,\varepsilon}$ in
  $D(\mu_{\varepsilon} h, C h^2)$ consists of a family of discrete
  eigenvalues $(\lambda_j(h,\varepsilon))_{1 \leq j \leq N}$ with
  algebraic multiplicities 1 such that
$$\lambda_j(h,\varepsilon) = \mu_{\varepsilon}h + \left((2j-1)c_0(\varepsilon) + c_{1}(\varepsilon) \right)h^2 + o(h^2),$$
where $c_0(\varepsilon)$, $c_1(\varepsilon) \in \C$ satisfy
$$c_0(\varepsilon) = b_0^{-1} \left\vert \frac{1}{2} \nabla^2 B (0) \right\vert^{1/2} + \grandO(\varepsilon).$$
\end{corollary}

\begin{proof}
  For $\varepsilon >0$ small enough, $B + \varepsilon V$ admits a
  unique and non-degenerate real critical point $z_{\varepsilon}$
  close to $0$, since it is a perturbation of $B$. Then
  $p = B + \varepsilon V$ satisfies the hypotheses of
  Theorem~\ref{thm.nsa}, as a perturbation of $B$, and Corollary
  \ref{coro.perturbation} is just a reformulation of this result. Note that $h_0>0$ is independent of $\varepsilon$ because, as one can check in our proof, the constants involved in our estimates can be chosen uniform in $\varepsilon$ (in particular the distance $ch$ between the eigenvalues in Theorem \ref{Main-Thm} \eqref{thm.a}).
\end{proof}

\begin{remark}
  Our results are adapted to the `ground state' of $\mathscr{L}_{h}$,
  \emph{i.e.} eigenfunctions associated with eigenvalues whose
  location is dictated by the minimum of the function $F$. It would be
  interesting to try and adapt the method to treat excited states, for
  which we should be able to leverage Rouby's results about 1D
  non-selfadjoint pseudodifferential operators~\cite{rouby-17}.
\end{remark}

\subsection{Structure of the article}

The article is organized as follows. In Section \ref{sec.2}, using a phase space change of variables, we prove that $\mathscr{L}_h$
is unitarily equivalent to an operator $h \Loh$ that can be seen as a
perturbation of a harmonic oscillator. As a consequence, we obtain a
microlocalization of eigenfunctions at distance $\grandO(h^\delta)$,
$\delta\in\interval[open right]0{\tfrac{1}{2}}$, of the characteristic
manifold. In Section \ref{sec.3.Lh.hat}, we introduce a slight
modification of $\Loh$ by inserting microlocal cutoff functions in the
symbol. In Section~\ref{subsec.3.2} we see $\widehat{\mathscr{L}}_h$
as a pseudodifferential operator with operator-valued symbol, and we
expand its symbol in powers of $h^{1/2}$. The properties of its
principal symbol $\mathsf{P}_0$, which is essentially a harmonic
oscillator, are described in Section \ref{subsec.3.3}. In Section
\ref{sec.4.Parametrix}, we use a Grushin method to reduce the
spectral analysis to the one of the effective operator $\Peff$, and in
Section \ref{sec.5.Reduction} we prove that the spectrum of
$\widehat{\mathscr{L}}_h$ is approximated by the spectrum of
$\Peff$. In Section \ref{sec.6.remove.chi} we `remove the cutoff
functions', proving that the spectrum of $\Loh$ is close to the
spectrum of $\widehat{\mathscr{L}}_h$, thus concluding the proof of
Theorem \ref{Main-Thm}. Finally, in Section \ref{sec.7.Hitrik}, we
prove Theorem \ref{thm.nsa} by explaining how to describe the spectrum
and the resolvent of $\Peff$.

\section{A first conjugation}\label{sec.2}

Using the semiclassical Weyl quantization, we may view $\mathscr{L}_h$
as an $h$-pseudo-differential operator:
\[
  \mathscr{L}_h=\Op^w_h(H)\,,\quad H(q,p)= (p_1-A_1(q_1,q_2))^2 + (p_2
  - A_2(q_1,q_2))^2 + h V(q_1,q_2)\,.
\]
Microlocal analysis suggests that eigenfunctions of $\mathscr{L}_h$
for eigenvalues of order $\mathscr{O}(h)$ should be localized near the
characteristic manifold $H^{-1}(0)$.  The aim of this section is to
introduce new phase space coordinates $(x_1,x_2,\xi_1,\xi_2)$ for
which this characteristic manifold becomes linear, and such that the
coordinates $(x_1,\xi_1)$ represent the distance to the characteristic
manifold. Scaling these coordinates by the natural factor $\sqrt h$
will finally yield Proposition~\ref{prop.Lh0.hat} below. A similar
conjugation (without the scaling) was performed in~\cite{GBRVN21}.

Due to the gauge invariance, we may assume that $\mathbf{A}$ has the
form
\[\mathbf{A}=(0,A_2)\,,\quad
  A_2(q)=\int_0^{q_1}B(s,q_2)\mathrm{d}s\,.\]

The diffeomorphism $\varphi : \R^2\to\R^2$ defined
in~\eqref{eq.diffeo} now reads
\[ \varphi(q)= (A_2(q),\, q_2)\,.\] For any function $f:\R^2\to \C$,
we shall denote $\rond f=f\circ \varphi^{-1}$.
\begin{proposition}\label{prop.Lh0.hat}
  The operator $\mathscr{L}_h$ is unitarily equivalent to
  \[
    h \Loh := h \Op^{w,2}_{h} \Op_1^{w,1}
    \widehat{H}^0,
  \]
  where
  \begin{equation}\label{eq.H0.hat}
    \begin{split}
      \widehat{H}^0 (x,\xi) =&\rond{B}(\xi_2 + h^{1/2} x_1, x_2 + h^{1/2} \xi_1)^2 \xi_1^2 + \left( x_1 + \rond{\alpha}(\xi_2 + h^{1/2}x_1,x_2 + h^{1/2}\xi_1) \xi_1 \right)^2\\
      &+ \rond{V}(\xi_2 + h^{1/2}x_1, x_2 + h^{1/2}\xi_1) + h W(\xi_2 +
      h^{1/2}x_1, x_2 + h^{1/2}\xi_1)\,,
    \end{split}
  \end{equation}
  and
  \begin{enumerate}[\rm (i)]
  \item $\alpha(q_1,q_2) = \partial_2 A_2 (q_1,q_2)$,
  \item
    $W = \frac{1}{4} (\partial_1 \rond{B})^2 + \frac{1}{4} (\partial_1
    \rond \alpha)^2 \,.$
  \end{enumerate}
\end{proposition}

\begin{remark}
  Here $\Op^{w,2}_h$ is the $h$-Weyl quantization with respect to $(x_2,\xi_2)$, and $\Op^{w,1}_1$ is the non-semiclassical Weyl quantization with respect to $(x_1,\xi_1)$, which means
  \begin{multline}
  \Op^{w,2}_{h} \Op_1^{w,1} u(x_1,x_2) =  \\ \frac{1}{(2\pi)^2
    h}\int_{\R^4} e^{\frac{i}{h}(x_2-y_2)\xi_2+i(x_1-y_1)\xi_1}
  \widehat{H}^0(\frac{x_1+y_1}2, \frac{x_2+\xi_2}{2}) u(y_1,y_2)\mathrm{d} y_1
  \mathrm{d} y_2 \mathrm{d} \xi_1 \mathrm{d} \xi_2.
\end{multline}
  
\end{remark}

\subsection{Proof of Proposition \ref{prop.Lh0.hat} } 
We split the proof into two steps : Lemma \ref{lemma.tildeH} and Lemma \ref{lemma.Lh0}.

\begin{lemma}\label{lemma.tildeH}
  The operator $\mathscr{L}_h$ is unitarily equivalent to
  $\widetilde{\mathscr{L}}_h=\Op^w_h \tilde H$, where
\begin{equation}\label{equ:H-tilde}
\begin{split}
\tilde H (x,\xi) &= \rond{B}(\xi_2 + x_1, x_2 + \xi_1)^2 \xi_1^2 + \left( x_1 + \rond{\alpha}(\xi_2 + x_1,x_2 + \xi_1) \xi_1 \right)^2
+ h \rond{V}(\xi_2 + x_1, x_2 + \xi_1) \\ &+ \frac{h^2}{4} (\partial_1 \rond B (\xi_2 + x_1, x_2 + \xi_1))^2 + \frac{h^2}{4}(\partial_1 \rond \alpha (\xi_2 + x_1, x_2 + \xi_1))^2 \,.
\end{split}
\end{equation}
\end{lemma}

\begin{proof}
First, let us rewrite the operator in the variables $(x_1,x_2) = \varphi(q_1,q_2)$. For any $u \in \Cinf(\R^d)$ recall that we denote $\rond u = u \circ \varphi^{-1}$, so that $\rond u (x) = u(q)$. Then we have :
\[ \begin{cases}
\partial_{q_1} u = \rond B \partial_{x_1} \rond u,\\
\partial_{q_2} u = \rond \alpha \partial_{x_1} \rond u  + \partial_{x_2} \rond u \,,
\end{cases}
\] 
with $\alpha(q) = \partial_2 A_2(q)$. Then $\mathscr{L}_h$ is given in these variables by
\[
  \mathscr{L}_h u = -h^2 (\rond B \partial_{x_1})^2 \rond{u} + \left(
    -ih \rond \alpha \partial_{x_1} - i h \partial_{x_2} - x_1
  \right)^2 \rond u + h \rond V \rond u \,,
\]
because $x_1 = A_2(q)$. In
other words, if $U$ is the following unitary transformation:
\[
  U : \left\{ \begin{matrix}
      \Ld(\R^2) &\rightarrow &\Ld(\R^2)\\
      u &\mapsto & \vert \dd \varphi^{-1} \vert^{1/2} u \circ
      \varphi^{-1}
\end{matrix}  \right.
\]
then, since the Jacobian $\vert \dd \varphi \vert$ equals $B$,
\[
  U \mathscr{L}_h U^* \rond u = -h^2 \rond B ^{1/2} \partial_{x_1} \rond
  B \partial_{x_1} (\rond{B}^{1/2} \rond{u}) + \rond B ^{-1/2} \left( -ih
    \rond \alpha \partial_{x_1} - i h \partial_{x_2} - x_1 \right)^2
  (\rond B ^{1/2} \rond u ) + h \rond V \rond u \,.
\]
With the notation $hD_j = -ih \partial_{x_j}$ we can rewrite it as
\begin{equation}\label{eq.ULU}
 U \mathscr{L}_h U^*  = \left(\rond B^{1/2} h D_1 \rond B^{1/2} \right)^2 + \left( \frac{1}{2}(\rond \alpha D_1 + D_1 \rond \alpha) + D_2 - x_1 \right)^2 + h \rond V \,.
 \end{equation}
Indeed, this follows from
\[ B^{-1/2} ( \rond \alpha D_1 + D_2 - x_1) B^{1/2} = \frac{1}{2}(\rond \alpha D_1 + D_1 \rond \alpha) + D_2 - x_1  \,, \]
which one can get using $(D_1 \rond \alpha) = \rond B^{-1} ( \rond \alpha D_1 \rond B + D_2 \rond B)$. Now, the Weyl-symbol of \eqref{eq.ULU} is
\begin{equation*}
\check{H}(x,\xi) = \rond B(x)^2 \xi_1^2 + (\rond \alpha \xi_1 + \xi_2 - x_1)^2 + h \rond V + \frac{h^2}{4} (\partial_{x_1} \rond B)^2 + \frac{h^2}{2}(\partial_{x_1} \rond \alpha )^2 \,.
\end{equation*}
Finally we do the following linear canonical change of variables,
\[ \begin{cases}
\tilde{x}_1 &= x_1 - \xi_2\\
\tilde{\xi}_1 &= \xi_1 \\
\tilde{x}_2 &= x_2 - \xi_1 \\
\tilde{\xi}_2 &= \xi_2 \,.
\end{cases} \] Using the linear Egorov Theorem (metaplectic
representation~\cite[Theorem 18.5.9]{HormanderIII}), $\mathscr{L}_h$ is
unitarily equivalent to $\Op_h^w \tilde{H}$ with:
\[ \tilde{H}(\tilde{x}_1, \tilde{x}_2, \tilde{\xi}_1, \tilde{\xi}_2) =
  H(\tilde{x}_1 + \tilde{\xi}_2, \tilde{x}_2 + \tilde{\xi}_1,
  \tilde{\xi}_1, \tilde{\xi}_2) \,, \] and the lemma is proved.
\end{proof}

\begin{lemma}\label{lemma.Lh0}
  The operator $\widetilde{\mathscr{L}}_h= \Op_h^w \tilde{H} $ is
  unitary equivalent to $h \Loh$.
\end{lemma}

\begin{proof}
First we split the $x_1$ and $x_2$ quantizations: 
\[ \widetilde{\mathscr{L}}_h = \Op_h^{w,2} \Op_h^{w,1} \tilde{H}\,. \]
Then we can change the semiclassical quantization with respect to
$(x_1,\xi_1)$ into a non-semiclassical one. Removing the
$(x_2,\xi_2)$-dependence in the notations, we have:
\[ \left( \Op_h^{w,1} \tilde{H} \right) u (x_1) = \frac{1}{2 \pi h }
  \int e^{\frac{i}{h}(x_1 - y_1) \xi_1} \tilde{H} \Big{(} \frac{x_1 +
    y_1}{2}, \xi_1\Big{)}u(y_1) \dd y_1 \dd \xi_1 \,.\] We do the
following change of variables
\begin{equation}
  x_1 = h^{1/2} \hat x_1, \, y_1 = h^{1/2} \hat y_1,\, \xi_1 =
  h^{1/2} \hat \xi_1,\label{equ:changement-hat}
\end{equation}
\[ \left( \Op_h^{w,1} \tilde{H} \right) u (x_1) = \frac{1}{2\pi} \int e^{(\hat x _1 - \hat y_1) \hat \xi_1} \tilde{H} \Big{(} h^{1/2} \frac{\hat x_1 + \hat y_1}{2}, h^{1/2} \hat \xi_1 \Big{)} u ( h^{1/2} \hat y_1 ) \dd \hat y_1 \dd \hat \xi_1 \,.\]
If we denote by $V$ the unitary transformation $Vu(\hat x_1) = u( h^{1/2} \hat x_1) h^{1/4}$ then we deduce
\[ \left( \Op_h^{w,1} \tilde{H} \right) u = V^* \Op_1^{w,1} a V u \]
with $a(x_1,\xi_1,h) = \tilde{H}(h^{1/2} x_1, h^{1/2} \xi_1)$. Note that $a = h \widehat{H}^0$ to conclude the proof.
\end{proof}

\begin{remark}
  In section \ref{subsection.symbol.class} we show that
  $(x_2,\xi_2) \mapsto \Op_1^{w,1} \widehat H^0$ belongs to a suitable
  class of operator-valued symbols, and hence can be seen as the
  operatorial symbol of $\widetilde{\mathscr{L}}_h$.
\end{remark}

We now check that the diffeomorphism $\varphi$ behaves well with
respect to the symbol classes.
\begin{lemma}
For any order function $m$, and any function $f : \R^2 \rightarrow \C$,
\[f \in S(m) \Rightarrow \rond{f} \in S(\rond{m})\,.\]
In particular, if $f \in S(1)$, then $\rond{f} \in S(1)$.
\end{lemma}

\begin{proof}
The derivatives of $\rond{f}$ are related to the derivatives of $f$ by:
\[
\begin{cases}
\partial_1 \rond{f} = \frac{1}{B} \partial_1 f\\
\partial_2 \rond{f} =-\frac{\alpha}{B}\partial_1 f+ \partial_2 f.
\end{cases}
\]
Iterating this formula, and using $B \geq b_0$ for the denominators and $B \in S(1)$ for the numerators, we deduce that:
$$ \vert \partial^{\gamma} \rond{f} (\rond{q}) \vert \leq C_{\gamma} \sum_{\vert \ell \vert \leq \vert \gamma \vert} \vert \partial^{\ell} f(q) \vert,$$
and thus if $f \in S(m)$,
$$ \vert \partial^{\gamma} \rond{f}(\rond{q}) \vert \leq C_{\gamma} m(q) = C_{\gamma} \rond{m}(\rond{q}).$$
\end{proof}

\subsection{Microlocalization of the eigenfunctions}

Using the quadratic behaviour of the symbol $\tilde{H}$ in the
variable $X_1:=(x_1,\xi_1)$ (Equation~\eqref{equ:H-tilde}), we prove
here that the eigenfunctions of $\widetilde{\mathscr{L}}_h$
corresponding to eigenvalues that are $\grandO(h^2)$-close to
$\mu_0 h$ are microlocalized in a band of width
$(\abs{X_1}\leq C h^\delta)$, for all
$\delta\in\interval[open right]0{\frac12}$.

\begin{lemma}\label{lem.microlocalisation}
  Let $\delta \in \left[0,\frac12\right)$ and
  $\chi \in \Cinf_0(\R,\R)$ equal to $1$ on a neighborhood of
  $0$. Then, for any normalized eigenpair $(\lambda, \psi)$ of
  $\widetilde{\mathscr{L}}_h$ with $\lambda \in D(\mu_0h, Ch^2)$:
  \[
    \psi = \Op_h^w \chi(h^{-\delta} x_1) \chi(h^{-\delta} \xi_1) \psi
    + \mathscr{O}_{L^2(\R^2)}(h^{\infty})\,.\]
\end{lemma}
\begin{proof}
  Let us start by proving the result when $\delta=0$, which is a crude
  microlocalisation. We let
  $\chi_0(x,\xi) = 1 - \chi( x_1) \chi(\xi_1) \in S_{\R^4}(1)$, and
  $\chi_0^w = \Op_h^w \chi_0$. We want to prove that
  $\Vert \chi_0^w \psi \Vert = \grandO(h^{\infty})$. Consider
  \begin{equation}\label{eq.micro01}
    \widetilde{\mathscr{L}}_h \chi_0^w \psi=
    \chi_0^w \widetilde{\mathscr{L}}_h \psi+  \left[ \widetilde{\mathscr{L}}_h , \chi_0^w \right] \psi\,.
  \end{equation}
  Since $\widetilde{\mathscr{L}}_h \psi = \lambda\psi$, we have
  \begin{equation}\label{eq.micro02}
    \| \chi_0^w \widetilde{\mathscr{L}}_h \psi\| \leq \tilde Ch  \Vert \chi_0^w \psi \Vert \,.
  \end{equation}
  If follows from the symbolic calculus (see for
  instance~\cite[Theorem 4.18]{Zworski}) that
  $\left[ \widetilde{\mathscr{L}}_h , \chi_0^w \right]$ is a
  pseudodifferential operator in $hS(1)$. By the Calderón-Vaillancourt
  Theorem, and taking $\underline{\chi}_0$ a cutoff function with the
  same properties of $\chi_0$, which is equal to $1$ on
  $\supp \chi_0$, we have
  \begin{equation}\label{eq.micro03}
    \left\| \left[ \widetilde{\mathscr{L}}_h, \chi_0^w \right] \psi\right\|
    \leq Ch \Vert \underline{\chi}_0^w \psi \Vert
    + \grandO(h^\infty) ||\psi||^2 \,.
  \end{equation}
  (In this text, we use the phrase ``cutoff function'' for smooth
  functions, independent of $h$, taking values in $\interval01$, whose
  support is not necessarily compact.)  From~\eqref{equ:H-tilde}, the
  symbol $\tilde{H}$ of $\widetilde{\mathscr{L}}_h$ satisfies
  \begin{equation*}
    \Re  \tilde{H} \geq c (\xi_1^2 +x_1^2) - Ch\,.
  \end{equation*}
  Let us consider a cutoff function $\chi_1\in S_{\R^2}(1)$ equal to 1
  in a neighborhood of the origin $X_1=0$, and such that, viewed as a
  function of $X\in \R^4$, we have
  $\supp\chi_1\cap\supp{\chi_0}=\emptyset$. We let
  $a_{h}(X):= \tilde H(X)+\chi_1(X_1)$. Then, for some $\tilde c>0$
  and $h$ small enough,
  \[\Re a_h(X)\geq \tilde c\langle X_1\rangle^2\geq \tilde c>0\,.\]
  We have $\frac{1}{a_{h}}\in S(\langle X_1\rangle^{-2})$. With the
  Calder\'on-Vaillancourt theorem, it follows that
  \[\forall\psi\in L^2(\R^2)\,,\quad \|[a_{h}^{-1}]^w\psi\|\leq
    C\|\psi\|\,.\] Thus,
  \[\|[a_{h}^{-1}]^w [a_{h}]^w\psi\|\leq C\|[a_{h}]^w\psi\|\,.\] By
  using again the symbolic calculus,
  \[
    \|\psi\|\leq \tilde C\|[a_{h}]^w\psi\|\,.
  \]
  By using that the supports of $\chi_0$ and $\chi_1$ are disjoint, we
  deduce that
  \[c \Vert \chi_0^w \psi \Vert \leq \| \widetilde{\mathscr{L}}_h
    \chi_0^w \psi\|+\mathscr{O}(h^{\infty})\|\psi\| \,. \] With
  \eqref{eq.micro01}, \eqref{eq.micro02}, and \eqref{eq.micro03}, we
  deduce that
  \[
    c \Vert \chi_0^w \psi \Vert\leq Ch\Vert \underline{\chi}_0^w \psi
    \Vert + \grandO(h^\infty)\|\psi\|\,.
  \]
  Iterating with $\underline{\chi}_0$ instead of $\chi_0$, we get
  $\Vert \chi_0^w \psi \Vert = \grandO(h^{\infty})$ and this concludes
  the proof in the case when $\delta=0$.

  Let us now consider the case $\delta\in\left(0,\frac12\right)$. We
  write again
  \[\tilde H^w \psi = \lambda\psi\,.\]
  Thanks to the rough microlocalization of the eigenfunctions
  established when $\delta=0$, up to a remainder
  $\mathscr{O}(h^\infty)$, we can replace $\tilde H$ by a symbol
  $\check H$ in $S(1)$ and that coincides with $\tilde H$ on
  $\{|X_1|\leq M\}$ for some arbitrary $M>0$ and that satisfies
  \begin{equation}
    \label{eq:lower-bound-X1}
    \Re  \check{H} \geq cp(X_1) - Ch\,,
  \end{equation}
  where $p(X_1)$ equals $|X_1|^2$ near $(0,0)$ and is constant away
  from a neighborhood of $(0,0)$.

  We have
  \[\check H^w\psi=\lambda\psi+R^w_h\psi\,,\quad
    R_h^w\psi=\mathscr{O}(h^\infty)\|\psi\|\,.\] We let
  $\chi_{0,\delta}(x,\xi) = 1 - \chi( h^{-\delta}x_1)
  \chi(h^{-\delta}\xi_1)$. We write again
  \begin{equation}\label{eq.micro01'}
    \check H^w \chi_{0,\delta}^w \psi =
    \chi_{0,\delta}^w \check H^w \psi +
    \left[ \check H^w , \chi_{0,\delta}^w \right] \psi\,.
  \end{equation}
  By the symbolic calculus in $S^\delta(1)$, we see that the symbol of
  $\left[ \check H^w , \chi_{0,\delta}^w \right]$ belongs to
  $h^{1-2\delta}S^\delta(1)$ and, due to the quadratic behaviour in
  $X_1$, actually belongs to $hS^\delta(1)$. Similarly to the case
  $\delta=0$, we get
  \begin{equation}\label{eq.micro03'}
    \left\| \left[ \check H^w, \chi_{0,\delta}^w \right] \psi\right\| \leq Ch \Vert \underline{\chi}_{0,\delta}^w \psi \Vert   + \grandO(h^\infty) \|\psi\|^2 \,.
  \end{equation}
  It is known that, modulo $\grandO(h^\infty)$, the Weyl quantization
  is unitarily equivalent to a positive quantization, namely the
  Toeplitz quantization on the Bargmann space; we denote by $\Op^+_h$
  the corresponding positive quantization on $L^2(\R^2)$. Let $q$ be
  the Toeplitz symbol of $\check H$, so that
  \[
    \Op_h^w \check H = \Op_h^+ q + \grandO(h^\infty).
  \]
  Since $\check H\in S(1)$, we have $q\in S(1)$ and
  $q = \check H + \grandO(h)$, see for instance~\cite[Theorem
  13.10]{Zworski}.

  Now, we consider a smooth cutoff function
  $\check\chi\in \mathscr{C}_0^\infty(\R^2)$ equal to 1 in a
  neighborhood of the origin, whose support is disjoint from the
  support of $\chi_0$, viewed as a function of $X_1$. Let
  $\check\chi_\delta(X_1):=\check\chi(X_1/h^\delta)$, and
  $\check a_{h,\delta}(X):= q(X) +
  h^{2\delta}\check\chi_{\delta}(X_1)$. For some $c>0$ and $h$ small
  enough,~\eqref{eq:lower-bound-X1} gives
  \[
    \Re \check a_{h,\delta}(X)\geq cp(X_1) + ch^{2\delta}\geq c
    h^{2\delta}\,.
  \]
  Hence $\Re\Op_h^+\check a_{h,\delta}\geq c h^{2\delta}$, in the
  sense of selfadjoint operators. Let $\tilde\chi_{\delta}$ be the
  Weyl symbol of $\Op_h^+\check\chi_{\delta}$, so that
  $\Op^+_h \check a_{h,\delta} = \Op_h^w(\check H +
  h^{2\delta}\tilde\chi_{\delta}) + \grandO(h^\infty)$. Using
  again~\cite[Theorem 13.10]{Zworski}, we see that
  $\tilde\chi_{\delta}\in S^\delta(1)$ and takes real values. While
  $\tilde\chi_{\delta}$ cannot vanish on any open set (it is
  analytic), it admits as asymptotic expansion in powers of
  $h^{1-2\delta}$ in the topology of $S^\delta(1)$, whose support is,
  for all fixed $h$, contained in the support of
  $\check\chi_\delta$. In particular
  \begin{equation}
    \tilde\chi_{\delta}^w \chi_{0,\delta}^w = \grandO(h^\infty).
    \label{eq:dis}
  \end{equation}

  Since $\Im (\check H + h^{2\delta}\tilde \chi_\delta) = \grandO(h)$,
  we have
  \[
    ch^{2\delta}\|\varphi\|\leq \| (\check
    H^w+h^{2\delta}\tilde\chi_{\delta}^w)\varphi\| +
    \grandO(h^\infty)\norm{\varphi}\,.
  \]
  Taking $\varphi=\chi_{0,\delta}^w\psi$ and using
  \eqref{eq.micro01'}, \eqref{eq.micro03'} and \eqref{eq:dis}, we find
  that
  \[ch^{2\delta}\|\chi_{0,\delta}^w\psi\|\leq
    Ch\|\underline{\chi}_{0,\delta}^w\psi\|+\mathscr{O}(h^{\infty})\|\psi\|\,.\]
  Since $2\delta<1$, an induction argument (on the size of the support
  of $\chi_{0,\delta}$) gives the result.
\end{proof}

\begin{remark}
  In this proof we have used a detour via the Toeplitz quantization
  because, using the standard $S^\delta$ symbolic calculus, if a
  symbol $a(x,\xi)\in S^\delta$ satisfies $a\geq ch^{2\delta}$, then
  the (precise) G\aa rding inequality only implies that
  \[
    \Op_h^w (a) \geq ch^{2\delta} - \grandO(h^{1-2\delta})~.
  \]
  Therefore, the proof of the lemma would require
  $\delta<\tfrac{1}{4}$. It turns out that one can get a
  better result while staying with the Weyl quantization, using the
  special form of the Fefferman-Phong inequality due to
  Bony~\cite{Bony99}, as follows.

  We first write
  \[
    \mathrm{Op}^w_h a=\mathrm{Op}^w_1
    a(x,h\xi)\,,
  \]
  which is unitarily equivalent to
  $\mathrm{Op}^w_1 a(h^{\frac12}x,h^{\frac12}\xi)$. We have
  \[
    a(h^{\frac12}x,h^{\frac12}\xi)- ch^{2\delta}\geq 0\,.
  \]
  Then, we notice that $a(h^{\frac12}\cdot)$ belongs to $S(1)$ since
  $\delta\in\interval[open right]0{\tfrac12}$. In fact, we even have,
  for all $\gamma\in\N^2$ with $|\gamma|\geq 4$,
  \[|\partial^{\gamma} b_{h,\delta}(h^{\frac12}\cdot)|\leq
    c_\gamma\,,\quad
    b_{h,\delta}(\cdot)=h^{-2+2\delta}(a(\cdot)-\tilde
    ch^{2\delta})\,.
  \]
  Because of this, the Bony-Fefferman-Phong inequality states that,
  for all $\varphi\in L^2(\R^2)$,
  \[
    \langle\mathrm{Op}_1^w
    b_{h,\delta}(h^{\frac12}\cdot)\varphi,\varphi\rangle\geq
    -C\|\varphi\|^2\,,
  \] and thus, after rescaling,
  \[
    \Re\langle\mathrm{Op}_h^w
    a\,\varphi,\varphi\rangle\geq (\tilde
    ch^{2\delta}-Ch^{2-2\delta})\|\varphi\|^2\,.
  \]
  We now see that $\delta<\frac12$ is enough to obtain
  \[
    \langle\mathrm{Op}_h^w
    a\,\varphi,\varphi\rangle\geq
    ch^{2\delta}\|\varphi\|^2\,,\]

\end{remark}

\section{The truncated operator $\widehat{\mathscr{L}}_h$}\label{sec.3.Lh.hat}

It follows from Lemma~\ref{lem.microlocalisation} that eigenpairs
$(\lambda,\varphi)$ of $\Loh$ with
$\lambda \in D(\mu_0, Ch)$ satisfy
\[
  \varphi =( \Op_1^w \chi_{\delta}) \varphi + \grandO_{L^2(\R^2)}(h^{\infty}) \,,
\]
with
$\chi_{\delta}(\hat x_1,\hat \xi_1) = \chi(h^{\frac{1}{2}-\delta} \hat
x_1) \chi(h^{\frac{1}{2}-\delta} \hat \xi_1)\in S(1)$, see
also~\eqref{equ:changement-hat}.  This motivates the introduction of a
so-called \textit{truncated} operator $\widehat{\mathscr{L}}_h$, whose
spectrum, as we shall prove in Section \ref{sec.6.remove.chi}, will be
close to the spectrum of $\Loh$ in the desired region.

\begin{definition}\label{def.Lhat}
  Fix $\delta \in \interval[open]0{\frac12}$ and let $\chi$ be a
  smooth cutoff function on $\R^2$, supported in a small neighborhood
  of $0$, and equal to $1$ near $0$. The truncated operator
  $\widehat{\mathscr{L}}_h$ is the pseudodifferential operator
$$\widehat{\mathscr{L}}_h = \Op_h^{w,2} \Op_1^{w,1} p_h$$
with symbol (see~\eqref{eq.H0.hat})
\begin{multline}\label{symbol-ph}
  p_h ( x, \xi)
  = \rond{B}^2(\xi_2 + h^{\frac12}\chi_\delta x_1, x_2 +
  h^{\frac12}\chi_\delta \xi_1) \xi_1^2 + \left( x_1 + \rond{\alpha}(\xi_2
    + h^{\frac12} \chi_\delta x_1,x_2
    +  h^{\frac12}\chi_\delta \xi_1)  \xi_1 \right)^2\\
  + \rond{V}(\xi_2 + h^{\frac12} \chi_\delta x_1, x_2 + h^{1/2}
  \chi_\delta \xi_1) + h W(\xi_2 + h^{\frac12} \chi_\delta x_1, x_2 +
  h^{\frac12} \chi_\delta \xi_1) \,,
\end{multline}
where $\chi_\delta =\chi(h^{\frac12-\delta} (x_1,\xi_1))$.
\end{definition}

Thanks to this cutoff function, we will expand $\rond{B}$ and
$\rond \alpha$ with respect to $h^{1/2} \chi_{\delta} x_1$ without
increasing the powers of $(x_1,\xi_1)$ at infinity, hence remaining in
a suitable class of symbols. As in the previous section, we use the
notation $X_j=(x_j, \xi_j)$, $j=1,2$, and $X=(X_1,X_2)\in \R^4$.

\subsection{Operator-valued symbol of
  $\widehat{\mathscr{L}}_h$}
\label{subsection.symbol.class}

We now focus on
$\widehat{\mathscr{L}}_h=\Op^{w,2}_h\Op^{w,1}_1(p_h)$. Note that, due
to our assumptions, we have for some $c_1,c_2>0$,
\[c_1( |X_1|^2 +1) \leq p_h\leq c_2 (|X_1|+1)^2\,.\] 

\begin{notation}
  We consider the operator symbol of $\widehat{\mathscr{L}}_h$ defined
  by
  \[\mathsf{P}_h(X_2)=\Op^{w,1}_1(p_h)\,,\]
  which for each fixed $X_2\in\R^2$ acts on the domain
  \[B^2(\R)=\{\psi\in H^2(\R) : x^2\psi\in L^2(\R) \}\,.\]
\end{notation}

\begin{lemma}\label{lem.ferme}
  For all $X_2\in\R^2$, the operator $(B^2(\R), \mathsf{P}_h(X_2))$ is
  closed. Its graph norm is equivalent to $\|\cdot\|_{B^2(\R)}$
  (uniformly in $X_2\in\R^2$ and $h>0$ small enough). In particular,
  $(B^2(\R), \mathsf{P}_h(X_2))$ has compact resolvent.
\end{lemma}

\begin{proof}
It is enough to prove the following two inequalities,
\begin{equation}\label{eq.B2.01}
\Vert \mathsf{P}_h(X_2) \psi \Vert \leq C \Vert (1+ \vert X_1 \vert^2)^w \psi \Vert\,,
\end{equation}
\begin{equation}\label{eq.B2.02}
\Vert (\mathsf{P}_h(X_2) + 1) \psi \Vert \geq c \Vert ( \vert X_1 \vert^2)^w \psi \Vert \,, 
\end{equation}
for all $\psi \in \mathscr{S}(\R)$, and for some positive constants $C$ and $c$ independent of $h$ and $X_2$. Note that
\begin{equation}
  \vert \partial_{X_1}^{\gamma} p_h \vert \leq c_\gamma (\vert X_1
  \vert^2 + 1)\label{equ:d-gamma-p-X1}
\end{equation}
for some constant $c_\gamma$ independent of
$(h,X_2)$. In other words, $p_h$ belongs to
the symbol class $S(1+ \vert X_1 \vert^2)$ uniformly with respect to
$(h,X_2)$. Thus the Weyl product
$p_h \star (1+ \vert X_1 \vert^2)^{-1}$ belongs to $S(1)$ uniformly
with respect to $(h,X_2)$, and by using Calderón-Vaillancourt theorem,
we get
\[\Vert \mathsf{P}_h(X_2) \left[(1+ \vert X_1 \vert^2)^{-1}\right]^w \varphi \Vert \leq C \Vert \varphi \Vert \,, \quad \forall \varphi \in \mathscr{S}(\R)\,, \]
with $C>0$ independent of $(h,X_2)$, and \eqref{eq.B2.01} follows. Actually, $p_h + 1$ is also elliptic in $S(1+ \vert X_1 \vert^2)$ uniformly with respect to $(X_2,h)$:
\[ \exists c_0 >0 \,, \quad \vert p_h + 1 \vert \geq c_0 (1+ \vert X_1 \vert^2)\,.\]
Hence $(p_h +1) \star (1+ \vert X_1 \vert^2)^{-1}$ is elliptic in $S(1)$ and the parametrix construction implies
\[ \forall \varphi \in \mathscr{S}(\R)\,, \quad  \Vert (\mathsf{P}_h(X_2)+1) \left[(1+ \vert X_1 \vert^2)^{-1}\right]^w \varphi \Vert \geq c \Vert \varphi \Vert\,, \]
and \eqref{eq.B2.02} follows.
\end{proof}

Let us consider the class of "bounded" operator-valued symbols (see, for instance, \cite[Chapitre 2]{Keraval} or \cite{Martinez07}):
\begin{multline*}
S(\R^2,\mathcal{L}(B^2(\R), L^2(\R)))\\
=\Big\{\mathsf{P}\in\mathscr{C}^{\infty}(\R^2,\mathcal{L}(B^2(\R), L^2(\R))) : \forall\alpha\in\N^2\,, \exists C_\alpha>0\,,\forall X_2\in\R^2\,,\\
\forall\psi\in B^2(\R)\,,\|\partial^\alpha_{X_2}\mathsf{P}(X_2)\psi\|\leq C_\alpha\|\psi\|_{B^2(\R)} \Big\}\,.
\end{multline*}
Since the inequality~\eqref{equ:d-gamma-p-X1} still holds when $p_h$
is replaced by its $X_2$-derivatives, we get the following.
\begin{lemma}\label{lemma.symbol.class}
The operator symbol $X_2\mapsto \mathsf{P}_h(X_2)$ belongs to $S(\R^2,\mathcal{L}(B^2(\R), L^2(\R)))$.
\end{lemma}

\subsection{Expansion of the symbol of $\widehat{\mathscr{L}}_h$}
\label{subsec.3.2}

We now impose $\delta \in (\frac 1 3, \frac 1 2)$. We prove the
following expansion for the symbol $\mathsf{P}_h = \Op_1^{w,1} p_h$ of
$\widehat{\mathscr{L}}_h$, where the condition $\delta>\frac{1}{3}$
ensures that the remainder $h^{3\delta}R_h$ is indeed negligible with
respect to the other terms.

\begin{lemma}
We have
  \[\mathsf{P}_h= \mathsf{P}_0 + h^{1/2} \mathsf{P}_1 + h \mathsf{P}_2
    +h^{3 \delta}R_h \] 
    for some symbols $\mathsf{P}_0$, $\mathsf{P}_1$, $\mathsf{P}_2$ defined in \eqref{eq.Pj}, and $R_h \in S(\R^2,\mathcal{L}(B^2(\R),L^2(\R)))$ uniformly bounded with respect to $h$. Moreover the
  principal symbol $\mathsf{P}_0(X_2)$ is the following
  $X_2$-dependent "harmonic oscillator":
  $$\mathsf{P}_0(X_2) = \Op_1^{w,1} \left( \rond B^2(\xi_2,x_2)\xi^2_1+(x_1+\rond\alpha(\xi_2,x_2)\xi_1)^2+ \rond V(\xi_2,x_2) \right).$$
\end{lemma}

\begin{proof}
  From Formula~\eqref{symbol-ph} we notice that $p_h$ can be seen as a
  smooth function $p_h=\check p(X_1,X_2,\hbar,\chi_\delta)$, where
  $\hbar=h^{1/2}$.  We may Taylor expand the symbol $\check p$ with
  respect to the third variable ``$\hbar$'':
\[
  \check p(X_1,X_2,\hbar,\chi_\delta)=p_0+\hbar p_1+ \hbar^2 p_2 +\hbar^{3} r_\hbar\,,
\]
where
\[
  r_\hbar=\frac{1}{2}\int_0^1(1-t)^2 \partial^{3}_\hbar \check
  p(X_1,X_2,t\hbar,\chi_\delta)\dd t\,,
\]
and
$p_j=\frac{1}{j!}\partial^j_{\hbar} \check
p(X_1,X_2,0,\chi_\delta)$. Note that the $p_j$ still slightly depend
on $h$ (through the cutoff functions $\chi_\delta$). Explicitely, we
have
\[\begin{split}
p_0&=\rond B^2(\xi_2,x_2)\xi^2_1+(x_1+\rond\alpha(\xi_2,x_2)\xi_1)^2+ \rond V(\xi_2,x_2)\,,\\
p_1&=\chi_\delta\left[2\xi^2_1\rond B\nabla\rond B\cdot X_1+2\xi_1(x_1+\rond\alpha\xi_1)\nabla\rond\alpha\cdot X_1+ \nabla \rond V\cdot X_1\right]  \,,\\
p_2&=\chi^2_\delta\left[\xi_1^2(\nabla\rond B\cdot X_1)^2+\xi_1^2\nabla^2\rond B(X_1,X_1)+\xi^2_1(\nabla\rond\alpha\cdot X_1)^2+\xi_1(x_1+\rond\alpha\xi_1)\nabla^2\rond\alpha(X_1,X_1)\right]\\
&+\chi^2_\delta \nabla^2 \rond V(X_1,X_1)+ W(\xi_2,x_2)\,,
\end{split}\]
where the functions $\rond B$, $\rond \alpha$, $\rond V$ and their gradients are implicitely taken at $(\xi_2,x_2)$. Letting 
\begin{equation}\label{eq.Pj}
\mathsf{P}_j(X_2)=\Op^{w,1}_1 p_j \,,
\end{equation}
we notice that
\[\mathsf{P}_h(X_2) =  \mathsf{P}_0(X_2) + \hbar \mathsf{P}_1(X_2) + \hbar^2 \mathsf{P}_2(X_2) + \hbar^{3}\Op^{w,1}_1 r_\hbar\,.\]
By using the Calderón-Vaillancourt theorem, due to the cutoff
functions we can check that
\[\forall \psi\in B^2(\R)\,,\quad \|\partial^\alpha_{X_2}\Op^{w,1}_1 r_\hbar\psi\|\leq C_{\alpha}h^{-3 \left(\frac12-\delta\right)}\|\psi\|_{B^2(\R)}\,.\]
Therefore, we can write
\[\mathsf{P}_h= \mathsf{P}_0 + \hbar \mathsf{P}_1 + \hbar^2 \mathsf{P}_2 +\hbar^{6 \delta}R_\hbar \]
with $R_\hbar\in S(\R^2,\mathcal{L}(B^2(\R),L^2(\R))) $ uniformly
bounded with respect to $h$. This leads to choosing
$\delta\in\left(\frac13,\frac12\right)$ and concludes the proof.

\end{proof}

\subsection{About the principal part $\mathsf{P}_0(X_2)$}
\label{subsec.3.3}

In this section we describe some important properties of
$\mathsf{P}_0(X_2)$. Since it is a `harmonic oscillator', \emph{i.e.}
the quantization of a positive definite quadratic form in $X_1$, we
have the following spectral properties. If we let
\[f_{X_2}(x_1)=C(X_2)e^{ - \frac{\rond{B}}{2( \rond{B}^2 + \rond{\alpha}^2)}  x_1^2 }e^{i\frac{\rond\alpha}{2(\rond B^2+\rond\alpha^2)}x_1^2}\,,\quad C(X_2)=\left(\frac{\rond B}{\pi (\rond B^2+\rond\alpha^2)}\right)^{\frac14}\,.\]
then we have
\[\mathsf{P}_0(X_2)f_{X_2}=\mu(X_2)f_{X_2}\,,\quad \mu(X_2)=\rond B(\xi_2,x_2)+\rond V(\xi_2,x_2)\,.\]
Moreover, the eigenvalues of $\mathsf{P}_0(X_2)$ are in the form 
\[(2n-1)\rond B(\xi_2,x_2)+ \rond V(\xi_2,x_2)\,,\quad n\geq 1\,.\]
Thus $f_{X_2}$ is the ground state of $\mathsf{P}_0(X_2)$. When considering the restriction of $\mathsf{P}_0(X_2)$ to $f_{X_2}^{\perp}$, which is a stable subspace, we get the following lemma.

\begin{lemma}\label{lem.bijpiperp}
  The operator $(\mathsf{P}_0(X_2)-z) : B^2(\R)\cap f_{X_2}^\perp \to f_{X_2}^\perp$ is bijective when $$u \Re (z-\mu_0) + v \mathsf{Im} (z-\mu_0) < 2ub_0,$$ and in particular when $z \in D(\mu_0,Ch)$ if $h$ is small enough.
\end{lemma}
\begin{proof}
We notice that, for all $\psi\in B^2(\R)\cap f_{X_2}^\perp$,
\begin{small}
\begin{align*}
\Re \left[ (u-iv) \langle (\mathsf{P}_0(X_2)-z) \psi,\psi \rangle \right] &\geq \left[ u (3\rond B (X_2) + \Re (\rond V(X_2) - z)) + v \mathsf{Im} ( \rond V(X_2) - z) \right] \Vert \psi \Vert^2 \\
&\geq \left[ 2 u b_0 + u \Re(\mu_0 - z) + v \mathsf{Im} (\mu_0 -z) \right] \Vert \psi \Vert^2,
\end{align*}
\end{small}
where we used that $u (\rond B (X_2) + \Re \rond V(X_2)) + v \mathsf{Im} \rond V(X_2) \geq \mu_0$ (Assumption \ref{Assumption}). This shows that $\mathsf{P}_0(X_2)-z$ is injective with closed range. But $\mathsf{P}_0(X_2)-z$ is a selfadjoint harmonic oscillator (up to an additive constant), so the conclusion follows.	
\end{proof}

\begin{proposition}\label{prop.matrix.P0}
Let $z\in D(\mu_0, Ch)$. We consider the operator
\[\mathscr{P}_0(X_2,z)=\begin{pmatrix}
\mathsf{P}_0(X_2)-z&\cdot f_{X_2}\\
\langle\cdot, f_{X_2}\rangle&0
\end{pmatrix}\in\mathcal{L}(B^2(\R)\times\C,L^2(\R)\times\C)\,.\]
We have
\[\mathscr{P}_0(\cdot,z)\in S(\R^2,\mathcal{L}(B^2(\R)\times\C,L^2(\R)\times\C))\,.\]
Moreover, if $h$ is small enough, $\mathscr{P}_0(X_2,z)$ is a bijection and
\[\mathscr{Q}_0:=\mathscr{P}_0^{-1}=\begin{pmatrix}
(\mathsf{P}_0-z)^{-1}\Pi^{\perp}&\cdot f_{X_2}\\
\langle\cdot, f_{X_2}\rangle&z-[\rond B+ \rond V]
\end{pmatrix}\in S(\R^2,\mathcal{L}(L^2(\R)\times\C,B^2(\R)\times\C))
\, ,\]
where $\Pi^{\perp} =\mathrm{Id}-\langle\cdot, f_{X_2}\rangle f_{X_2}.$
\end{proposition}
\begin{proof}
Let us consider $(\psi,\beta)\in L^2(\R)\times\C$ and look for $(\varphi,\alpha)\in B^2(\R)\times\C$ such that
\[(\mathsf{P}_0-z)\varphi=\psi-\alpha f_{X_2}\,,\quad \langle\varphi,f_{X_2}\rangle=\beta\,.\]
The first equation has solutions only if \[\langle \psi-\alpha f_{X_2},f_{X_2}\rangle=(\mu-z)\langle\varphi,f_{X_2}\rangle=(\mu-z)\beta\,,\]
where $\mu=\mu(X_2)= \rond{B}(X_2)+ \rond{V}(X_2)$;
this is equivalent to $\alpha=\langle\psi,f_{X_2}\rangle+(z-\mu)\beta$. With this choice, we write
\[(\mathsf{P}_0-z)(\varphi-\beta f_{X_2})=\psi-\alpha f_{X_2}+\beta(\mu-z)f_{X_2}\in f_{X_2}^\perp\,.\]
It remains to apply Lemma \ref{lem.bijpiperp}.
\end{proof}

\section{Parametrix construction and consequences}\label{sec.4.Parametrix}
\subsection{Parametrix construction}
Let us now consider the \enquote{Grushin operator symbol}
\[X_2\mapsto\mathscr{P}_h(X_2)=\begin{pmatrix}
\mathsf{P}_h(X_2)-z&\cdot f_{X_2}\\
\langle\cdot,f_{X_2}\rangle&0
\end{pmatrix}\in S(\R^2,\mathcal{L}(B^2(\R)\times\C, L^2(\R)\times\C))\,,\]
and notice that
\begin{equation}\label{eq.41}
\Op^{w,2}_h\mathscr{P}_h=\begin{pmatrix}
\widehat{\mathscr{L}}_h-z&\mathfrak{P}^*\\
\mathfrak{P}&0
\end{pmatrix}\,,\quad \mathfrak{P}=\Op^{w,2}_h\langle\cdot, f_{X_2}\rangle\,.
\end{equation}
Note that for $\delta \in (\frac13, \frac12)$,
\[\mathscr{P}_h(X_2)=\mathscr{P}_0+h^{1/2}\mathscr{P}_1+h\mathscr{P}_2+ \grandO(h^{3\delta})\,,\]
with $\mathscr{P}_0$ defined in Proposition \ref{prop.matrix.P0}, and for $j\geq 1$,
\[\mathscr{P}_j=\begin{pmatrix}
\mathsf{P}_j(X_2)&0\\
0&0
\end{pmatrix}\,,\]
and the remainder belongs to $S(\R^2,\mathcal{L}(B^2(\R)\times\C, L^2(\R)\times\C))$. The following proposition is an approximate parametrix construction.
\begin{proposition}\label{prop.parametrix}
For $z \in D(\mu_0, Ch)$ we consider
\[\begin{cases}
\mathscr{Q}_1=-\mathscr{Q}_0\mathscr{P}_1\mathscr{Q}_0 \\
\mathscr{Q}_2=-\mathscr{Q}_0\mathscr{P}_2\mathscr{Q}_0-\mathscr{Q}_1\mathscr{P}_1\mathscr{Q}_0-\frac{1}{2i}\{\mathscr{Q}_0,\mathscr{P}_0\}\mathscr{Q}_0\,,
\end{cases}\]
and we let
\[\mathscr{Q}^{[2]}_h=\mathscr{Q}_0+h^{1/2}\mathscr{Q}_1+h \mathscr{Q}_2\,.\]
Then we have
\[\left(\Op^{w,2}_h\mathscr{Q}^{[2]}_h\right)\left(\Op^{w,2}_h\mathscr{P}_h\right)=\mathrm{Id}+\mathscr{R}_{h,z}\,,\quad \mathscr{R}_{h,z}=\mathscr{O}(h^{3 \delta})\,,\]
where the bounded operator $\mathscr{R}_{h,z}$ depends on $z$ analytically. We also have
\[\left(\Op^{w,2}_h\mathscr{P}_h\right)\left(\Op^{w,2}_h\mathscr{Q}^{[2]}_h\right)=\mathrm{Id}+\widetilde{\mathscr{R}}_{h,z}\,,\quad \widetilde{\mathscr{R}}_{h,z}=\mathscr{O}(h^{3\delta})\,,\]
Moreover, the operator $\mathscr{Q}^{[2]}_h$ has the form
$$\mathscr{Q}^{[2]}_{h}(X_2) = \begin{pmatrix}
* & * \\
* & \mathscr{Q}^{[2]}_{h,\pm}(X_2)
\end{pmatrix},$$
with the scalar function
\begin{multline*}
\mathscr{Q}^{[2]}_{h,\pm}=z-\mu(X_2)-h^{1/2} \langle \mathsf{P}_1(X_2)f_{X_2}, f_{X_2}\rangle\\
+h\left(-\langle \mathsf{P}_2(X_2)f_{X_2},f_{X_2}\rangle+\langle\mathsf{P}_1(X_2)(\mathsf{P}_0(X_2)-z)^{-1}\Pi^{\perp}\mathsf{P}_1(X_2) f_{X_2},f_{X_2}\rangle\right)\,.
\end{multline*}
\end{proposition}

\begin{proof}
The composition $\Op^{w,2}_h\mathscr{Q}^{[2]}_h\Op^{w,2}_h\mathscr{P}_h$ gives a new pseudo-differential operator (with operator symbol). This symbol is given by the usual $h$-Moyal composition law (with $\hbar = h^{1/2}$)
$\mathscr{Q}^{[2]}_h\star \mathscr{P}_h$, and we have
\[\begin{split}
\mathscr{Q}^{[2]}_h\star \mathscr{P}_h=&(\mathscr{Q}_0+\hbar\mathscr{Q}_1+\hbar^2\mathscr{Q}_2)\star(\mathscr{P}_0+\hbar\mathscr{P}_1+\hbar^2\mathscr{P}_2+ \grandO(\hbar^{6\delta})\\
=&\mathscr{Q}_0\mathscr{P}_0+\frac{\hbar^2}{2i}\{\mathscr{Q}_0,\mathscr{P}_0\}+\hbar\left(\mathscr{Q}_0\mathscr{P}_1+\mathscr{Q}_1\mathscr{P}_0\right)\\
&+\hbar^2\left(\mathscr{Q}_2\mathscr{P}_0+\mathscr{Q}_0\mathscr{P}_2+\mathscr{Q}_1\mathscr{P}_1\right)+\mathscr{O}(\hbar^{6 \delta})\,.
\end{split}
\]
This leads to the formulas of the $\mathscr{Q}_j$. Let us now compute $\mathscr{Q}^{[2]}_{h,\pm}$ (slightly departing from the traditional "$-+$" subscript in Grushin problems, we use the subscript $\pm$ for the lower-right coefficient of block matrices). An easy product of operator matrices gives
\[\mathscr{Q}_{1,\pm}=-\langle \mathsf{P}_1(X_2)f_{X_2}, f_{X_2}\rangle\,.\]
In the same way, we get
\[\left(-\mathscr{Q}_0\mathscr{P}_2\mathscr{Q}_0\right)_{\pm}=-\langle \mathsf{P}_2(X_2)f_{X_2},f_{X_2}\rangle\,.\]
From
\[-\mathscr{Q}_1\mathscr{P}_1\mathscr{Q}_0=\mathscr{Q}_0\mathscr{P}_1\mathscr{Q}_0\mathscr{P}_1\mathscr{Q}_0\,,\]
we obtain
\[\left(-\mathscr{Q}_1\mathscr{P}_1\mathscr{Q}_0\right)_{\pm}=\langle\mathsf{P}_1(\mathsf{P}_0-z)^{-1}\Pi^{\perp}\mathsf{P}_1 f_{X_2},f_{X_2}\rangle\,.\]	
Using
\[\{\mathscr{Q}_0,\mathscr{P}_0\}=\partial_\xi\mathscr{Q}_0\partial_x\mathscr{P}_0-\partial_x\mathscr{Q}_0\partial_\xi\mathscr{P}_0\,,\]
another computation gives
\[\begin{split}
\left(\{\mathscr{Q}_0,\mathscr{P}_0\}\mathscr{Q}_0\right)_{\pm}
=&\langle(\partial_x \mathsf{P}_0)f_{X_2},\partial_\xi f_{X_2}\rangle-\langle(\partial_\xi \mathsf{P}_0)f_{X_2},\partial_x f_{X_2}\rangle
\\
&+(\mu-z)\left(\langle\partial_x f_{X_2},\partial_\xi f_{X_2}\rangle-\langle\partial_\xi f_{X_2},\partial_x f_{X_2}\rangle\right)\,.
\end{split}\]
Note that, since $f_{X_2}$ is an $L^2$-normalized eigenfunction,
\begin{multline*}
0=\partial_x\langle(\mathsf{P}_0-z)f_{X_2},\partial_\xi f_{X_2}\rangle=\langle(\partial_x \mathsf{P}_0)f_{X_2},\partial_\xi f_{X_2}\rangle+\langle (\mathsf{P}_0-z)\partial_x f_{X_2},\partial_\xi f_{X_2}\rangle\\
+\langle(\mathsf{P}_0-z)f_{X_2},\partial_x\partial_\xi f_{X_2}\rangle\,,
\end{multline*}
and
\begin{multline*}
	0=\partial_\xi\langle(\mathsf{P}_0-z)f_{X_2},\partial_x f_{X_2}\rangle=\langle(\partial_\xi \mathsf{P}_0)f_{X_2},\partial_x f_{X_2}\rangle+\langle (\mathsf{P}_0-z)\partial_\xi f_{X_2},\partial_x f_{X_2}\rangle\\
	+\langle(\mathsf{P}_0-z)f_{X_2},\partial_x\partial_\xi f_{X_2}\rangle\,,
\end{multline*}
Thus,
\begin{equation*}
	\left(\{\mathscr{Q}_0,\mathscr{P}_0\}\mathscr{Q}_0\right)_{\pm}=-\langle(\mathsf{P}_0-\mu)\partial_x f,\partial_\xi f\rangle+\langle(\mathsf{P}_0-\mu)\partial_\xi f,\partial_x f\rangle=0\,,
	\end{equation*}
where we used that $\mathsf{P}_0-\mu$ is selfadjoint.

Therefore,
\[\mathscr{Q}_{2,\pm}=-\langle \mathsf{P}_2(X_2)f_{X_2},f_{X_2}\rangle+\langle\mathsf{P}_1(X_2)(\mathsf{P}_0-z)^{-1}\Pi^{\perp}\mathsf{P}_1(X_2) f_{X_2},f_{X_2}\rangle\,.\]
\end{proof}

\begin{remark}
The expression of the \enquote{Schur complement} $\mathscr{Q}^{[2]}_{h,\pm}$ has already appeared in previous works (see \cite[Proposition 3.1.10]{Keraval}). Note however that the assumption of \cite[Hypothèse 3.1.9]{Keraval} (\emph{i.e.}, the principal operator symbol does not depend on $\xi$) is not satisfied in our context. It is also important to notice that our $h$-pseudo-differential operator is expanded according to the powers of $\hbar=h^{\frac12}$ and not $h$. This avoids the nasty Poisson brackets computations of \cite[Lemme 3.1.11]{Keraval}.
\end{remark}
The following lemma shows that the apparent subprincipal symbol of $\mathscr{Q}^{[2]}_{h,\pm}$ actually vanishes (modulo  $\mathscr{O}(h^\infty)$). More precisely:
\begin{lemma}\label{lem.P1.flat}
We have, in the symbol class $S(1)$,
\[\langle \mathsf{P}_1(X_2)f_{X_2}, f_{X_2}\rangle=\mathscr{O}(h^\infty)\,.\]	
\end{lemma}
\begin{proof}
This follows from the fact that $\mathsf{P}_1(X_2)=\Op^{w,1}_1 (\chi_{\delta} \tilde p_1)$ where
\[\tilde p_1= 2\xi^2_1\rond B\nabla\rond B\cdot X_1+2\xi_1(x_1+\rond\alpha\xi_1)\nabla\rond\alpha\cdot X_1+\nabla V\cdot X_1 \,.\]
Indeed, $ \tilde p_1$ is a homogeneous function of $X_1$ of odd order, and $f_{X_2}$ is an even function of $x_1$ so that, for all $X_2\in\R^2$,
\[\langle\Op^{w,1}_1\tilde p_1 f_{X_2},f_{X_2}\rangle=0\,.\]
From this, we see that the term $\chi_\delta$ will only contribute to $\mathscr{O}(h^\infty)$ due to the expo
nential decay of $f_{X_2}$. The same argument applies to the derivatives with respect to $X_2$.
\end{proof}

\begin{corollary}\label{cor.E}
For $h$ small enough, $\mathrm{Id}+\mathscr{R}_{h,z}$ and $\mathrm{Id}+\widetilde{\mathscr{R}}_{h,z}$ are bijective. If we let
\[
\mathscr{E}_h=(\mathrm{Id}+\mathscr{R}_{h,z})^{-1}\Op^{w,2}_h\mathscr{Q}^{[2]}_h\,,\quad
\widetilde{\mathscr{E}}_h=\Op^{w,2}_h\mathscr{Q}^{[2]}_h(\mathrm{Id}+\widetilde{\mathscr{R}}_{h,z})^{-1}\,,
\]	
we have
\[\mathscr{E}_h\cdot\Op^{w,2}_h\mathscr{P}_h=\mathrm{Id}\,,\quad \Op^{w,2}_h\mathscr{P}_h\cdot\widetilde{\mathscr{E}}_h=\mathrm{Id}\,,\quad \widetilde{\mathscr{E}}_h=\mathscr{E}_h\,.\]
\end{corollary}

\subsection{From $\widehat{\mathscr{L}}_h$ to $\mathscr{E}_{h,\pm}$}

According to Corollary \ref{cor.E}, the operator
$\mathscr{E}_{h} = \mathscr{E}_{h} (z)$ is the inverse of
$\Op_h^{w,2} \mathscr{P}_h$. We can write it in the matrix form
\[ \mathscr{E}_h = \begin{pmatrix}
\mathscr{E}_{h,++} & \mathscr{E}_{h,+} \\
\mathscr{E}_{h,-} & \mathscr{E}_{h,\pm} 
\end{pmatrix} \,.\]
Then we have the following classical observation (see \cite{SZ07}, for instance, for a review on Grushin methods).

\begin{lemma}\label{lem.E+-}
For $z \in D(\mu_0, Ch)$ we have, for $h$ small enough,
\[z\in\mathrm{sp}(\widehat{\mathscr{L}}_h)\Longleftrightarrow 0\in\mathrm{sp}(\mathscr{E}_{h,\pm}) \,.\]	
Moreover, when $z\notin\mathrm{sp}(\widehat{\mathscr{L}}_h)$, the following formulas hold
\begin{equation}
\mathscr{E}_{h,\pm}^{-1}=-\mathfrak{P}(\widehat{\mathscr{L}}_h-z)^{-1}\mathfrak{P}^*\,,
\end{equation}
and
\begin{equation}\label{eq.resE+-}
(\widehat{\mathscr{L}}_h-z)^{-1}=\mathscr{E}_{h,++}-\mathscr{E}_{h,+}\mathscr{E}^{-1}_{h,\pm}\mathscr{E}_{h,-}\,.
\end{equation}
\end{lemma}

\begin{proof}
From Corollary \ref{cor.E}, and in view of \eqref{eq.41}, we have
\begin{equation}\label{eq.Gru-a}
(\widehat{\mathscr{L}}_h-z){\mathscr{E}}_{h,+}+\mathfrak{P}^*{\mathscr{E}}_{h,\pm}=0\,,\quad \mathfrak{P}{\mathscr{E}}_{h,+}=\mathrm{Id}\,,
\end{equation}
and
\begin{equation}\label{eq.Gru-b}
\mathscr{E}_{h,-}(\widehat{\mathscr{L}}_h-z)+\mathscr{E}_{h,\pm}\mathfrak{P}=0\,,\quad \mathscr{E}_{h,-}\mathfrak{P}^*=\mathrm{Id}\,.
\end{equation}
By using \eqref{eq.Gru-a} and \eqref{eq.Gru-b}, we see that when $\widehat{\mathscr{L}}_h-z$ is bijective, so is $\mathscr{E}_{h,\pm}$. Then, assume that $\mathscr{E}_{h,\pm}$ is bijective and  consider also
\[\mathscr{E}_{h,++}(\widehat{\mathscr{L}}_h-z)+\mathscr{E}_{h,+}\mathfrak{P}=\mathrm{Id}\,.\]
With \eqref{eq.Gru-b}, we get
\[\left(\mathscr{E}_{h,++}-\mathscr{E}_{h,+}\mathscr{E}^{-1}_{h,\pm}\mathscr{E}_{h,-}\right)(\widehat{\mathscr{L}}_h-z)=\mathrm{Id}\,.\]
Using \eqref{eq.Gru-a} and also
\[(\widehat{\mathscr{L}}_h-z)\mathscr{E}_{h,++}+\mathfrak{P}^*\mathscr{E}_{h,-}=\mathrm{Id}\,,\]
we get
\[(\widehat{\mathscr{L}}_h-z)\left(\mathscr{E}_{h,++}-\mathscr{E}_{h,+}\mathscr{E}^{-1}_{h,\pm}\mathscr{E}_{h,-}\right)=\mathrm{Id}\,.\]
\end{proof}

\section{Spectral reduction}\label{sec.5.Reduction}

In Lemma \ref{lem.E+-} we proved that the spectrum of $\widehat{\mathscr{L}}_{h}$ in $D(\mu_0,Ch)$ is given by those $z$ such that $\mathscr{E}_{h,\pm}(z)$ is not bijective. Moreover, according to Corollary \ref{cor.E} and Proposition \ref{prop.parametrix}, $\mathscr{E}_{h,\pm}(z)=\Op_h^{w,2}\mathscr{Q}^{[2]}_{h,\pm}+\mathscr{O}(h^{3\delta})$, and hence
\begin{equation}\label{eq.mu-eff0}
\mathscr{E}_{h,\pm}(z) = \Op_h^{w} \left( z - \mu^{\mathsf{eff}}_h(X_2) + \grandO(h^{3\delta}) \right) \,. 
\end{equation}
with the effective symbol, belonging to $S_{\R^2_{X_2}}(1)$, given by
\begin{multline}\label{eq.mu-eff}
\mu^{\mathsf{eff}}_h(X_2) = \rond B (X_2) + \rond{V}(X_2) \\
+ h \left( \langle \mathsf{P}_2(X_2) f_{X_2}, f_{X_2} \rangle - \langle \mathsf{P}_1(X_2) (\mathsf{P}_0(X_2) - z)^{-1} \Pi^{\perp} \mathsf{P}_1(X_2) f_{X_2}, f_{X_2} \rangle \right) \,.
\end{multline}

Indeed, the $h^{1/2}$-order term appears to be small by Lemma \ref{lem.P1.flat}. We recall that $\mathsf{P}_j$ was defined in \eqref{eq.Pj}, $f_{X_2}$ is the first eigenfunction of the harmonic oscillator $\mathsf{P}_0(X_2)$ and $\Pi^{\perp}$ is the orthogonal projection onto $f_{X_2}^{\perp}$ (see Lemma \ref{lem.bijpiperp}). We denote $\Peff = \Op_h^w \mu^{\mathsf{eff}}_h$. The aim of this section is to prove that the spectrum of $\widehat{\mathscr{L}}_h$ is given by the spectrum of $\Peff$ up to a small error.

\subsection{The spectrum of $\widehat{\mathscr{L}}_h$ is discrete}

\begin{proposition}\label{prop.fred0}
The following families are analytic families of Fredholm operators of index $0$:
$$\begin{matrix}
\left(z-\Op^{w}_h(\rond B+\rond V)\right)_{z\in D(\mu_0,Ch)} &
\left(z-\Op^{w}_h(\mu^{\mathrm {eff}}_h)\right)_{z\in D(\mu_0,Ch)}\\
\left(\mathscr{E}_{h,\pm}(z)\right)_{z\in D(\mu_0,Ch)}&
\left(\widehat{\mathscr{L}}_h-z\right)_{z\in D(\mu_0,Ch)}
\end{matrix}$$
\end{proposition}

\begin{proof}
Let us consider the family $\left(z-\Op^{w}_h(\rond B+\rond V)\right))_{z\in D(\mu_0,Ch)}$. This is an analytic family of bounded operators. By perturbation, it is enough to prove that $\Op^{w}_h(\rond B+\rond V)-\mu_0$ is a Fredholm operator with index $0$ (since the set of Fredholm operators of index $0$ is open). Let $u >0$ and $v \in \R$ be given by Assumption \ref{Assumption}. The function $F = u(\rond B + \Re \rond V) + v \mathsf{Im} \rond V$ admits a global minimum, and there exist a compact $K$ and a constant $\gamma > \min F$ such that,
\[ \forall X_2 \in \R^2 \setminus K, \quad F(X_2) \geq \gamma \,. \]
Thus we may consider a smooth cutoff function $\chi$ supported in a neighborhood of $K$ such that 
\[ F+\chi \geq \gamma > \min F \,.\]
Define
\[ P = \Op_h^w ( \rond B + \rond V) + (u-iv)^{-1} \chi \,. \]
Proving that $P-\mu_0$ is invertible is enough to conclude that $\Op_h^w ( \rond B + \rond V) - \mu_0$ is a Fredholm operator with index $0$. Let $w= (u-iv) \mu_0$ and $Q=(u-iv)P$, so that
\begin{align*}
(u-iv)(P-\mu_0) &= Q - w\\
&=\Op_h^w \left( u \Re (\rond B +  \rond V - \mu_0) + v \mathsf{Im} ( \rond V  -  \mu_0) + \chi \right)\\ & \quad + i \ \Op_h^w \left( u \mathsf{Im}( \rond V - \mu_0 ) - v \Re( \rond V + \rond B - \mu_0) \right) \,.
\end{align*}
Each parenthesis being selfadjoint, we deduce for all $\psi \in \Ld(\R)$ that
\begin{align*}
\vert \langle (Q-w) \psi, \psi \rangle \vert &\geq \Re \langle (Q-w) \psi, \psi \rangle \\
& \geq \left\langle \Op_h^w \big{(} u \Re (\rond B +  \rond V - \mu_0) + v \mathsf{Im} ( \rond V  -  \mu_0) + \chi \big{)} \psi , \psi \right\rangle \,.
\end{align*}
Using the G\aa rding inequality, and with $\min F = u \Re \mu_0 + v \mathsf{Im} \mu_0$, we get
\[ \vert \langle (Q-w) \psi, \psi \rangle \vert \geq (\gamma - \min F) \Vert \psi \Vert^2 \,. \]
Hence $Q-w$ is one-to-one with closed range. We can apply the same arguments for the adjoint of $Q-w$. We deduce that $Q-w$ is bijective, and so is $P- \mu_0$.\\

By \eqref{eq.mu-eff0} and \eqref{eq.mu-eff}, we have, for $z\in D(\mu_0,Ch)$, $$\mathscr{E}_{h,\pm}(z)=\Op_h^w \left( z-\rond B - \rond V \right)+\mathscr{O}(h).$$ Thus, $\mathscr{E}_{h,\pm}(z)$ is Fredholm of index $0$, as soon as $h$ is small enough (and by construction it is analytic with respect to $z$). The same perturbation argument hold for $z- \Op_h^w \mu_h^{\mathsf{eff}}$. Using again Corollary \ref{cor.E}, this implies that $\widehat{\mathscr{L}}_h-z$ is also a Fredholm operator of index $0$ (the Fredholmness of the Schur complement $\mathscr{E}_{h,\pm}$ is equivalent to that of $\widehat{\mathscr{L}}_h-z$).
\end{proof}

Proposition \ref{prop.fred0} is not enough to establish that the spectrum of $\widehat{\mathscr{L}}_h$ is discrete in $D(\mu_0,Ch)$ : We have to check that the resolvent set intersects $D(\mu_0,Ch)$.

Thanks to the assumptions on $\Peff$, we can draw in the resolvent set
of $\Peff$ the circle $\Gamma_{j,h}$ of radius
$h^{\frac{3}{2}-\kappa}$ and center $\nu_j(h)$ for $h$ small enough.

\begin{lemma}\label{lem.resolvent}
  Let us denote by $D_{j,h}$ the open disc of center $\nu_j(h)$ and
  radius $h^{\frac32-\kappa}$ and let
  \[
    \mathcal{R}_h := D(\mu_0,Ch)\setminus \bigcup_{j\in\{1,\ldots,
      N\}}D_{j,h}\,.
  \]
  There exists $h_0>0$ such that, for all $h\in (0,h_0)$, we have
  \[
    \mathcal{R}_h \subset\rho(\widehat{\mathscr{L}}_h)\,.
  \]
\end{lemma}
\begin{proof}
  We recall that, for $z\in D(\mu_0,Ch)$, we have
  \[\mathscr{E}_{h,\pm}(z)=z-\Peff+\mathscr{O}(h^{3\delta})\,.\]
  By a classical perturbation argument using
  \eqref{eq.Res-bound}, we see that, for all $z\in \mathcal{R}_h$,
  $\mathscr{E}_{h,\pm}(z)$ is bijective. Indeed, it is bijective as
  soon as
  \[ \mathscr{O}(h^{3\delta}) \Vert (\Peff - z)^{-1} \Vert < 1\,, \]
  so it is sufficient that $h^{3 \delta} < h^{\frac{3}{2}-\kappa}$,
  which we enforce by taking, as we may,
  \begin{equation}
    \delta> \max(\tfrac{1}{2}-\tfrac{\kappa}{3}, \,\tfrac{1}{3})
    \label{equ:delta}
  \end{equation}
  (the lower bound $\delta>\frac{1}{3}$ comes from
  Section~\ref{subsec.3.2}). Thanks to Lemma \ref{lem.E+-}, we deduce
  that $\mathcal{R}_h\subset\rho(\widehat{\mathscr{L}}_h)$.
\end{proof}
	
Lemma \ref{lem.resolvent} and Proposition \ref{prop.fred0} imply that the spectrum of $\widehat{\mathscr{L}}_h$ in $D(\mu_0,Ch)$ is discrete (thanks to the analytic Fredholm theory). This also implies that, for $z\in D(\mu_0,Ch)$, $\mathscr{E}_{h,\pm}(z)$ is bijective except for discrete values of $z$.

\subsection{The spectrum of $\widehat{\mathscr{L}}_h$ lies near the one of $\Peff$}
The following proposition states that the spectrum of $\widehat{\mathscr{L}}_h$ must be located near the spectrum of the effective operator.
\begin{proposition}\label{prop.sp-close-speff}
There exist $h_0,\tilde C>0$ such that, for all $h\in(0,h_0)$,	if $\lambda\in D(\mu_0,Ch)\cap\mathrm{sp}(\widehat{\mathscr{L}}_h)$, then
\[ \mathrm{dist}(\lambda,\mathrm{sp}(\Peff))\leq \tilde C h^{\frac{3}{2} - \kappa} \,.\]
\end{proposition}
\begin{proof}
Since we know that the spectrum is discrete, we may consider an eigenpair $(\lambda,\psi)$. By Corollary \ref{cor.E}, we have
\[\mathscr{Q}_{h,++}^{[2]}(\widehat{\mathscr{L}}_h-\lambda)+\mathscr{Q}_{h,+}^{[2]}\mathfrak{P}=\mathrm{Id}+\mathscr{O}(h^{3\delta})\,,\]
and 
\[\mathscr{Q}_{h,-}^{[2]}(\widehat{\mathscr{L}}_h-\lambda)+\mathscr{Q}_{h,\pm}^{[2]}\mathfrak{P}=\mathscr{O}(h^{3\delta})\,,\]
so that
\[\|\psi\|\leq C\|\mathfrak{P}\psi\|\,,\quad\|\mathscr{Q}_{h,\pm}^{[2]}\mathfrak{P}\psi\|\leq Ch^{3\delta}\|\psi\|\leq \tilde Ch^{3\delta}\|\mathfrak{P}\psi\| \,.\]

The resolvent bound \eqref{eq.Res-bound} provides us with
\begin{equation}
\mathrm{dist}(\lambda,\mathrm{sp}(\Peff))\|\varphi\|\leq C\left\|\left(\lambda-\Peff \right)\varphi\right\|\,,
\end{equation}
and thus, we get
\[\mathrm{dist}(\lambda,\mathrm{sp}(\Peff))\|\mathfrak{P}\psi\|\leq Ch^{3\delta}\|\mathfrak{P}\psi\|\,.\]
\end{proof}

\subsection{The spectrum of $\Peff$ lies near the one of $\widehat{\mathscr{L}}_h$}
\begin{proposition}
  \label{prop:dimension1aumoins}
  Consider $j\in\{1,\ldots,N\}$. There exists $h_0>0$ such that, for
  all $h\in(0,h_0)$, the circle of center $\nu_j(h)$ and radius
  $h^{\frac{3}{2}-\kappa}$ encircles at least one point in the spectrum of
  $\widehat{\mathscr{L}}_h$.
\end{proposition}
\begin{proof}

We recall Lemma \ref{lem.resolvent}. For $z \in \Gamma_{j,h}$, thanks to a Neumann series we get
\begin{equation}\label{eq.E+-approx}
\mathscr{E}_{h,\pm,z}^{-1}=(z-\Peff)^{-1}(\mathrm{Id}+\mathscr{N}_{h,z})\,,
\end{equation}
where the bounded operator $\mathscr{N}_{h,z}$ satisfies 
\[\|\mathscr{N}_{h,z}\|\leq
  Ch^{3\delta}\mathrm{dist}(z,\mathrm{sp}(\Peff))^{-1}\leq \tilde C
  h^{3\delta-\frac{3}{2}+\kappa}<1 \,,\] uniformly with respect to
$z\in\Gamma_{j,h}$; the last inequality coming
from~\eqref{equ:delta}. Therefore, we get
\[\|\mathscr{E}_{h,\pm,z}^{-1}-(z-\Peff)^{-1}\|\leq \tilde C
  h^{3\delta-3+2\kappa}\,.\] Integrating over the contour (whose
length is $2\pi h^{\frac{3}{2}-\kappa}$), we find that
\[\left\|\frac{1}{2i\pi}
    \int_{\Gamma_{j,h}}\mathscr{E}_{h,\pm,z}^{-1}\dd z-\frac{1}{2i\pi}
    \int_{\Gamma_{j,h}}(z-\Peff)^{-1}\dd z\right\|\leq \tilde C
  h^{3\delta-\frac{3}{2}+\kappa}\,.\] We see that the right-hand-side
goes to $0$ when $h$ goes to $0$. We recall that
$\frac{1}{2i\pi}\int_{\Gamma_{j,h}}(z-\Peff)^{-1}\dd z$ is the (Riesz)
projection on the eigenspace of $\Peff$ associated with $\nu_j(h)$. If
$\Gamma_{j,h}$ does not encircles any element in the spectrum of
$\widehat{\mathscr{L}}_h$, we see with Lemma \ref{lem.E+-} that
$\frac{1}{2i\pi}\int_{\Gamma_{j,h}}\mathscr{E}_{h,\pm,z}^{-1}\dd z=0$,
and thus that the projection
$\frac{1}{2i\pi}\int_{\Gamma_{j,h}}(z-\Peff)^{-1}\dd z$ must be zero,
and this would be a contradiction.
\end{proof}

In fact, we can prove slightly more.
\begin{proposition}
	Let us consider the spectral projector $\Pi_{j,h}$ of $\widehat{\mathscr{L}}_h$ associated with the contour $\Gamma_{j,h}$. Then,
	\[\dim\mathrm{Ran}\,\Pi_{j,h}= 1 \,.\]
	In other words, there is exactly one eigenvalue of $\widehat{\mathscr{L}}_h$ encircled by $\Gamma_{j,h}$.
\end{proposition}

\begin{proof}
  We already know from Proposition~\ref{prop:dimension1aumoins} that
  $\dim\mathrm{Ran}\,\Pi_{j,h}\geq 1$. As in the proof of Proposition
  \ref{prop.sp-close-speff}, we have
\[
  \|\psi\|\leq C\|\mathfrak{P}\psi\| +
  C\|(\widehat{\mathscr{L}}_h-\nu_j(h))\psi\|+Ch^{3\delta}\|\psi\|\,,
\]
	and
	\[
          \|(\nu_j(h) -
          \Op^{w}_h\mu^{\mathrm{eff}}_h)\mathfrak{P}\psi\|\leq C
          h^{3\delta}\|\psi\|+C\|(\widehat{\mathscr{L}}_h -
          \nu_j(h))\psi\|\,.
        \]
        Let us assume that $\psi$ belong to the range of the
        projection
	\[\Pi_{j,h}=\frac{1}{2i\pi}\int_{\Gamma_{j,h}}(\eta-\widehat{\mathscr{L}}_h)^{-1}\dd\eta\,.\]
	We have
	\[\begin{split}
	(\widehat{\mathscr{L}}_h-\nu_j(h))\psi&=\frac{1}{2i\pi}\int_{\Gamma_{j,h}}(\widehat{\mathscr{L}}_h-\nu_j(h))(\eta-\widehat{\mathscr{L}}_h)^{-1}\psi\dd\eta\\
	&=\frac{1}{2i\pi}\int_{\Gamma_{j,h}}(\eta-\nu_j(h))(\eta-\widehat{\mathscr{L}}_h)^{-1}\psi\dd\eta\,,
	\end{split}\]
	so that
	\[\|(\widehat{\mathscr{L}}_h-\nu_j(h))\psi\|\leq h^{\frac32-\kappa} h^{\frac32-\kappa}\sup_{\eta\in\Gamma_{j,h}}\|(\eta-\widehat{\mathscr{L}}_h)^{-1}\|\|\psi\| \,.\]
	We recall \eqref{eq.E+-approx}, and notice that, for all $\eta\in \mathcal{R}_h$, in view of~\eqref{eq.Res-bound},
	\[\|\mathscr{E}_{h,\pm}^{-1}(\eta)\|\leq C \|(\eta-\Peff)^{-1}\|\leq C\mathrm{dist}(\eta,\mathrm{sp}(\Peff))^{-1}\,.\]
	With \eqref{eq.resE+-}, this gives
        \begin{equation}
        \|(\eta-\widehat{\mathscr{L}}_h)^{-1}\|\leq
        C\mathrm{dist}(\eta,\mathrm{sp}(\Peff))^{-1} \leq
        Ch^{\kappa-\frac{3}{2}}\,.\label{equ:res.hat.L}
      \end{equation}
	Thus, for all $\psi\in\mathrm{Ran}\,\Pi_{j,h}$,
	\[\|(\widehat{\mathscr{L}}_h-\nu_j(h))\psi\|\leq C h^{\frac32-\kappa}\|\psi\|\,.\]
	It follows that
		\begin{equation}\label{eq.injective}
		\|\psi\|\leq C\|\mathfrak{P}\psi\|\,,
	\end{equation}
	and
	\[\|(\nu_j(h)-\Peff)\mathfrak{P}\psi\|\leq Ch^{\frac32-\kappa}\|\psi\|\leq \tilde Ch^{\frac32-\kappa}\|\mathfrak{P}\psi\|\,.\]
	In particular, \eqref{eq.injective} implies that $\dim \mathfrak{P}(\mathrm{Ran}\,\Pi_{j,h})=\dim \mathrm{Ran}\,\Pi_{j,h}$. Then,
	for all $\varphi\in \mathfrak{P}(\mathrm{Ran}\,\Pi_{j,h})$,
	\[\|(\nu_j(h)-\Peff)\varphi\|\leq Ch^{\frac32-\kappa}\|\varphi\|\,.\]
	Let us now consider the spectral projection $\Pi^{\mathrm{eff}}_{j,h}$ associated with $\Peff$ and the contour $\Gamma_{j,h}$:
	\[\Pi^{\mathrm{eff}}_{j,h}=\frac{1}{2i\pi}\int_{\Gamma_{j,h}}(\eta-\Peff)^{-1}\dd\eta\,.\]
	In fact, since $\Gamma_{j,h}$ encircles only $\nu_j(h)$ as element of the spectrum of $\Peff$ and due to the gap of order $h$ between the eigenvalues of $\Peff$, we have also
	\[\Pi^{\mathrm{eff}}_{j,h}=\frac{1}{2i\pi}\int_{\tilde\Gamma_{j,h}}(\eta-\Peff)^{-1}\dd\eta\,,\]
	where $\tilde\Gamma_{j,h}$ is the circle of center $\nu_j(h)$
        and radius $h^{\frac{3}{2}-\tilde \kappa}$ where
        $\tilde\kappa> \kappa$.  We have
	\[
          \begin{split}
            \Pi^{\mathrm{eff}}_{j,h}\varphi & =\varphi+\frac{1}{2i\pi}\int_{\tilde\Gamma_{j,h}}\left[(\eta-\Peff)^{-1}-(\eta-\nu_j(h))^{-1}\right]\varphi\,\dd\eta\\
            & =\varphi+\frac{1}{2i\pi}\int_{\tilde\Gamma_{j,h}}(\eta-\Peff)^{-1}(\eta-\nu_j(h))^{-1}(\Peff-\nu_j(h))\varphi\,\dd\eta\,.
      \end{split}
	\]
        Since $\tilde\Gamma_{j,h}\subset\mathcal{R}_h$, we deduce that
	\[\|\Pi^{\mathrm{eff}}_{j,h}\varphi-\varphi\|\leq C
          h^{\tilde\kappa - \frac{3}{2}} h^{\tilde\kappa -
            \frac{3}{2}}
          h^{\frac32-\kappa}|\tilde\Gamma_{j,h}|\|\varphi\|\leq C
          h^{\tilde\kappa-\kappa}\|\varphi\| \,.\] In particular, for
        all $\varphi\in \mathfrak{P}(\mathrm{Ran}\,\Pi_{j,h})$,
	\[\|\varphi\|\leq C\|\Pi^{\mathrm{eff}}_{j,h}\varphi\|\,.\]
	This implies that
	\[1=\dim \mathrm{Ran}\, \Pi_{j,h}^{\mathrm{eff}}\geq \dim \mathrm{Ran}\,\Pi_{j,h}\,,\]
	and the conclusion follows.
\end{proof}

\section{Removing the cutoff function}\label{sec.6.remove.chi}

In the previous section, we proved that the spectrum of
$\widehat{\mathscr{L}}_h$ is close to the spectrum of
$\Peff$. $\widehat{\mathscr{L}}_h$ was the operator in which we
inserted cutoff functions $\chi_{\delta}$ (see Definition
\ref{def.Lhat}). Let us now remove these cutoff functions and prove
that the spectrum of the initial operator $\Loh$ (defined in
\eqref{eq.H0.hat}) is close to the spectrum of
$\widehat{\mathscr{L}}_h$ (see Proposition \ref{prop.Lh.to.Lh0}).

\begin{proposition}\label{prop.Lh.Fredholm}
The families $\left(\Loh-z\right)_{z\in D(\mu_0,Ch)}$ and $\left(\mathscr{L}_h-z\right)_{z\in D(\mu_0 h, C h^2)}$ are analytic families of Fredholm operators of index $0$. In particular the spectrum of $\Loh$ in $D(\mu_0,Ch)$ and of $\mathscr{L}_h$ in $D(\mu_0 h, C h^2)$ are discrete.
\end{proposition}

\begin{proof}
  By using the unitary equivalence of $h\Loh$ and
  $\mathscr{L}_h$, we can focus on the family
  $\left(\mathscr{L}_h-z\right)_{z\in D(\mu_0 h, Ch^2)}$. Let $u>0$
  and $v \in \R$ given by Assumption \ref{Assumption}. The function
  $F = u( B + \Re V) + v \mathsf{Im} V$ admits a global
  minimum, and there exists a compact $K$ and a constant
  $\gamma > \min F$ such that,
\[ \forall q \in \R^2 \setminus K, \quad F(q) \geq \gamma \,. \]
Thus we may consider a smooth cutoff function $\chi$ supported near $K$ such that
\[ F + \chi \geq \gamma > \min F \,. \]
Define
\[P = \mathscr{L}_h + h (u-iv)^{-1} \chi \,. \]
Proving that $P-hz$ is invertible is enough to conclude that $\mathscr{L}_h- hz$ is a Fredholm operator of index $0$. Let $w = (u-iv) z$ and $Q=(u-iv)P$, so that
\begin{align*}
(u-iv)(P-hz) &= Q-hw \\
&= \left( u ((ih \nabla + \A)^2 + h \Re (V-z)) + v h \mathsf{Im} (V-z) + h \chi \right)\\ 
& \quad + i \left( u h \mathsf{Im} (V-z) -v( (ih \nabla + \A)^2 + h \Re(V-z)) \right) \,.
\end{align*}
Each parenthesis being selfadjoint, we deduce for $\psi \in \mathsf{Dom}( \mathscr{L}_h)$ that
\begin{align*}
\vert \langle (Q-hw) \psi, \psi \rangle \vert &\geq \Re \langle (Q-hw) \psi, \psi \rangle \\
& \geq \left\langle \left( u(ih\nabla + \A)^2 + u h \Re(V-z) + v h \mathsf{Im}(V-z) + h\chi \right) \psi, \psi \right\rangle \,.
\end{align*}
Using the lower bound $(ih \nabla + \A)^2 \geq h B$, we get
\[ \vert \langle (Q-hw) \psi, \psi \rangle \vert \geq h \langle (u \Re(B + V - z) + v \mathsf{Im}(V-z) + \chi ) \psi, \psi \rangle \,. \]
For $z \in \C$ such that $u \Re (z- \mu_0) + v \mathsf{Im} (z- \mu_0) < C h$, since $\mu_0$ satisfies
\[ \min F = u \Re \mu_0 + v \mathsf{Im} \mu_0 \]
we have
\[ \vert \langle (Q-hw) \psi, \psi \rangle \vert \geq h (\gamma - \min F - C h) \Vert \psi \Vert^2 \,. \]
Hence $Q-hw$ is one-to-one with closed range. We can apply the same arguments for the adjoint of $Q-hw$. We deduce that $Q-w$ is bijective, and so is $P-hz$.\\

Thus $\mathscr{L}_h-hz$, for $z$ in
$\Omega = \lbrace z\in\C : u \Re (z- \mu_0) + v \mathsf{Im} (z- \mu_0) < C h
\rbrace$, is an analytic family of Fredholm operators with index
$0$. To conclude discreteness of the spectrum it remains to show that
$\Omega$ intersects the resolvent set of $\mathscr{L}_h$. To see this,
note that
$\Re(u-iv) \mathscr{L}_h \geq u h (B + \Re V) + vh \mathsf{Im} V \geq
u h \Re \mu_0 + vh \mathsf{Im} \mu_0$, and thus when
$\Re z \rightarrow - \infty$ (in $\Omega$) we must reach the resolvent set. The proposition follows since $D(\mu_0,Ch) \subset \Omega$.
\end{proof}

\begin{lemma}
There exists $h_0>0$ such that, for all $h\in(0,h_0)$ and all $\lambda\in D(\mu_0,Ch)\cap \mathrm{sp}(\Loh)$, we have $\lambda\in\cup_{j=1}^N D_{j,h}$. In particular, for all $j\in\{1,\ldots,N\}$, $\Gamma_{j,h}\subset\rho(\Loh)$.
\end{lemma}
\begin{proof}
Assume that it is not true. Then, for some $h$ (as small as desired), we can find an element of the spectrum $\lambda\in D(\mu_0,Ch)\setminus \cup_{j=1}^N D_{j,h}$ and it is a discrete eigenvalue according to Proposition \ref{prop.Lh.Fredholm}. Consider an associated normalized eigenfunction
\[\Loh\psi=\lambda\psi\,.\]
Using the microlocalization Lemma \ref{lem.microlocalisation} on the eigenfunctions of $\Loh$, we can add the cutoff functions $\chi_{\delta}$ in the symbol to get
\[\widehat{\mathscr{L}}_h\psi=\lambda\psi+\mathscr{O}(h^\infty)\,.\]
But we know from~\eqref{equ:res.hat.L} that the spectrum of
$\widehat{\mathscr{L}}_h$ inside $D(\mu_0,Ch)$ lies in
$\cup_{j=1}^N D_{j,h}$ and that the resolvent is controlled by a
negative power of $h$:
\[\|(\widehat{\mathscr{L}}_h-z)^{-1}\|\leq C h^{-\frac32+\kappa}\,,\]
for $z\in\mathcal{R}_h$.
This implies that $\psi=0$, and this contradicts the normalization of $\psi$.
\end{proof}
\begin{proposition}
For each $j\in\{1,\ldots, N\}$, the contour $\Gamma_{j,h}$ encircles at most one eigenvalue of $\Loh$ (with \textbf{geometric} multiplicity).
\end{proposition}
\begin{proof}
  If it is not the case, a contour $\Gamma_{j,h}$ encircles at least
  two eigenvalues $\lambda$ and $\mu$ associated with normalized
  orthogonal eigenfunctions $\varphi$ and $\psi$, respectively. We
  have $\hat \Pi_{j,h}^0\varphi=\varphi$ and
  $\hat \Pi_{j,h}^0\psi=\psi$. Then, the resolvent formula gives that
\[\begin{split}
\hat\Pi_{j,h}-\hat\Pi^0_{j,h}&=\frac{1}{2i\pi}\int_{\Gamma_{j,h}}\left((z-\widehat{\mathscr{L}}_h)^{-1}-(z-\Loh)^{-1}\right)\dd z\\
&=\frac{1}{2i\pi}\int_{\Gamma_{j,h}}(z-\Loh)^{-1}(\widehat{\mathscr{L}}_h-\Loh)(z-\widehat{\mathscr{L}}_h)^{-1}\dd z\,.\end{split}\]
By the microlocalization Lemma \ref{lem.microlocalisation} on $\varphi$ and $\psi$, we get:
\[\hat\Pi_{j,h}\varphi= \hat \Pi_{j,h}^0 \varphi+\mathscr{O}(h^\infty) = \varphi + \mathscr{O}(h^{\infty}) \,,\quad \hat\Pi_{j,h}\psi=\psi+\mathscr{O}(h^\infty) \,.\]
This implies that the range of $\hat\Pi_{j,h}$ is at least two, and this is a contradiction.
\end{proof}
In fact, we can even prove that each $\Gamma_{j,h}$ encircles exactly one eigenvalue (with \textbf{algebraic} multiplicity).
\begin{proposition}\label{prop.Lh.to.Lh0}
For each $j\in\{1,\ldots, N\}$, the contour $\Gamma_{j,h}$ encircles exactly one eigenvalue of $\Loh$ (with \textbf{algebraic} multiplicity).
\end{proposition}
\begin{proof}
The proof uses the ellipticity at infinity with respect to $X_1$. Let us consider a partition of the unity $\chi_{1,h}(X_1)+\chi_{2,h}(X_1)=1$ with $\mathrm{supp}\chi_{2,h}\subset\{|X_1|\geq h^{-\delta}\}$ and such that the operator $\chi^w_{1,h}(\widehat{\mathscr{L}}_h-\Loh)$ is $\mathscr{O}(h^\infty)$ (which is possible by definition of $\widehat{\mathscr{L}}_h$ and $\Loh$).

Let $N\in\N$. We have, for all $z\in\Gamma_{j,h}$, and all $v$,
\begin{equation}\label{eq.chi2h}
\|\chi^w_{2,h}(\Loh-z)^{-1}v\|\leq Ch^{2\delta}\|v\|+Ch^N\|(\Loh-z)^{-1}v\|\,.
\end{equation}
The estimate \eqref{eq.chi2h} follows by considering the equation
\[(\Loh-z)u=v\,,\]
and writing for instance that
\[(\Loh-z)\chi^w_{2,h}u=\chi^w_{2,h}v+[\chi^w_{2,h},\Loh]u\,,\]
so that
\[Ch^{-2\delta}\|\chi^w_{2,h}u\|\leq C\|v\|+Ch^{\delta}\|\underline{\chi}^w_{2,h}u\|+Ch^{N}\|u\|\,,\]
where the support of $\underline{\chi}_{2,h}$ is slightly larger than the one of $\chi_{2,h}$. By induction, we get \eqref{eq.chi2h}.
At this stage, we still do not control the whole resolvent $(\Loh-z)^{-1}$. By the resolvent formula, and the symbolic calculus, we see that
\[\chi^w_{1,h}\left[(\Loh-z)^{-1}-(\widehat{\mathscr{L}}_h-z)^{-1}\right]=\underbrace{\chi^w_{1,h}(\widehat{\mathscr{L}}_h-z)^{-1}\left[\widehat{\mathscr{L}}_h-\Loh\right]}_{=\mathscr{O}(h^{\infty})}(\Loh-z)^{-1}\,.\]
Therefore,
\begin{equation}\label{eq.chi1h}
\|\chi^w_{1,h}(\Loh-z)^{-1}v\|\leq \|(\widehat{\mathscr{L}}_h-z)^{-1}v\|+Ch^N\|(\Loh-z)^{-1}v\|\,.
\end{equation}
Combining \eqref{eq.chi2h} and \eqref{eq.chi1h}, we get
\[\|(\Loh-z)^{-1}v\|\leq C\|v\|+\|(\widehat{\mathscr{L}}_h-z)^{-1}v\|\leq \tilde C\left(1+\Vert (\widehat{\mathscr{L}}_h - z)^{-1} \Vert\right)\|v\|\,.\]
In particular, for all $z\in\Gamma_{j,h}$,
\[\|(\Loh-z)^{-1}\|\leq Ch^{-\frac32+\kappa}\,.\]
Coming back to \eqref{eq.chi2h}, we deduce that
\begin{equation}\label{eq.chi2h'}
\|\chi^w_{2,h}(\Loh-z)^{-1}\|\leq \tilde Ch^{2\delta}\,.
\end{equation}
Let us now estimate the difference of the spectral projections by
using the microlocal partition of the unity:
\begin{multline*}
\hat\Pi_{j,h}-\hat\Pi^0_{j,h}
=\frac{1}{2i\pi}\int_{\Gamma_{j,h}}\chi^w_{2,h}\left((z-\widehat{\mathscr{L}}_h)^{-1}-(z-\Loh)^{-1}\right)\dd z\\
+\frac{1}{2i\pi}\int_{\Gamma_{j,h}}\chi^w_{1,h}\left((z-\widehat{\mathscr{L}}_h)^{-1}-(z-\Loh)^{-1}\right)\dd z\,.
\end{multline*}
We get
\begin{equation*}
\|\hat\Pi_{j,h}-\hat\Pi^0_{j,h}\|
\leq C|\Gamma_{j,h}|h^{2\delta}+\frac{1}{2\pi}\left\|\int_{\Gamma_{j,h}}\underbrace{\chi^w_{1,h}(z-\widehat{\mathscr{L}}_h)^{-1}(\Loh-\widehat{\mathscr{L}}_h)}_{=\mathscr{O}(h^\infty)}\underbrace{(z-\Loh)^{-1}}_{=\mathscr{O}(h^{-\frac32+\kappa})}\dd z\right\|\,.
\end{equation*}
Thus, for $h$ small enough, we get
\[\|\hat\Pi_{j,h}-\hat\Pi^0_{j,h}\|<1\,,\]
and these projections have the same rank. In particular, the contour $\Gamma_{j,h}$ encircles as many eigenvalues (with algebraic multiplicity) of $\Loh$ as of $\widehat{\mathscr{L}}_h$ (\emph{i.e.}, exactly one).
\end{proof}

\section{On the spectrum of $\Peff$}\label{sec.7.Hitrik}
In this section, we give a
description of the spectrum in the disc $D(\mu_0,Ch)$ of $\Peff$ whose symbol is
\begin{equation*}
\mu^{\mathsf{eff}}_h(X) =\rond p(X)
+ h \rond p_1(X) \,,\quad X\in\R^2\,,
\end{equation*}
where
\[\rond p(X)= \rond B (X) + \rond{V}(X) \,,\]
and
\[\rond p_1(X)=\langle \mathsf{P}_2(X) u_{X}, u_{X} \rangle - \langle \mathsf{P}_1(X) (\mathsf{P}_0(X) - z)^{-1} \Pi^{\perp} \mathsf{P}_1(X) u_{X}, u_{X} \rangle\,.\]
We work under the assumptions of Theorem \ref{thm.nsa}. By using Assumption \ref{Assumption}, we may assume without loss of generality that $(u,v)=(1,0)$ and that $\Re \rond p$ has its unique minimum at $0$. In particular, we may write
\begin{equation}\label{eq.mueffexpansion}
\mu_h^{\mathrm{eff}}(X)=\mu_0+h\rond p_1(0)+Q_0(X)+\underbrace{R_3(X)+hR_1(X)}_{=R_h(X)}\,,
\end{equation}
with $R_1(X)=\mathscr{O}(|X|)$, $R_3(X)=\mathscr{O}(|X|^3)$, and 
\[Q_0(X)=\frac12\mathrm{Hess}\,\rond p(X,X)\,.\]
Under our assumptions, $\Re Q_0$ is positive. By translation, we may assume that $\mu_0=0$ and $\rond p_1(0)=0$.

\subsection{On the spectrum and resolvent of $Q_0^w$}
The spectrum and the resolvent of $Q_0^w$ are easy to describe in $D(0,Ch)$. We recall below these properties for the convenience of the reader. Some of the considerations below may be found in \cite{Davies99}, \cite[Chapter 14]{Helffer13}, \cite{Hitrik04}, \cite{Boulton02}, or \cite{PS06}.

\begin{proposition}
There exists $c_0\in\C^*$ such that for all $h>0$,
\[\mathrm{sp}(Q_0^w)=\{(2n-1)c_0 h\,, n\geq 1\}\,.\]
The spectrum is made of eigenvalues of algebraic multiplicity one.
Moreover there exists $D>0$ such that, for all $h>0$ and all $z\in D(0,Ch)\setminus\mathrm{sp}(Q_0^w)$,
\begin{equation}\label{eq.resQ0}
\|(z-Q_0^w)^{-1}\|\leq \frac{D}{\mathsf{dist}(z,\mathrm{sp}(Q_0^w))}\,.
\end{equation}
\end{proposition}
\begin{proof}
By using the homogeneity of $Q_0^w$ and the rescaling $x=h^{\frac12}y$, we may assume that $h=1$. Then, we write
\[Q_0=\Re Q_0+i\Im Q_0\,.\]
Since $\Re Q_0$ is positive, up to a linear symplectic transformation, we may assume (thanks to the metaplectic representation) that
\[Q_0(X)=c(x^2+\xi^2)+i Q_1(X)\,,\]
where $Q_1$ is a real quadratic form and $c>0$. Up to a Euclidean rotation, we may assume that $Q_1(X)=ax^2+b\xi^2$ with $(a,b)\in\R^2$. Thus,
\[Q_0(X)=(c+ia)x^2+(c+ib)\xi^2\,.\]
After dividing by $c+ib$ and rescaling, we are reduced to
\[Q_0(X)=\xi^2+e^{i\alpha}x^2\,,\qquad \alpha\in[0,\pi)\,.\] The
complex harmonic oscillator $Q_0^w$ has non-empty resolvent set (since
it is a sectorial operator) and compact resolvent\footnote{These
  elementary properties follow from the inequality
  \[\Re\left( e^{-i\frac\alpha2}\langle
      Q_0^w\psi,\psi\rangle\right)\geq\cos\left(\frac\alpha2\right)\left(\|\psi'\|^2+\|x\psi\|^2\right)\geq\cos\left(\frac{\alpha}{2}\right)\|\psi\|^2\,.\]}
(and thus its spectrum is discrete). Considering the classical Hermite
functions $f_n=e^{-\frac{x^2}{2}}P_n$, we see that the functions
$g_n(x)=f_n(e^{i\alpha/4}x)$ are eigenfunctions of $Q_0^w$ associated
with the eigenvalues $(2n-1)e^{i\frac\alpha2}$. Moreover, the closure
of ${\mathrm{span} (\overline{g_n}\,, n\in\N^*)}$ being $L^2(\R)$, we
easily see that, if $\lambda$ is eigenvalue, it must be in the form
$\lambda=(2n-1)e^{i\frac\alpha2}$. This gives the announced
description of the spectrum.

We can also check that the eigenvalues are geometrically simple by using Wronskian considerations. In fact, we can see that they are algebraically simple by using the analytic perturbation theory of Kato (see \cite[Chapter VII, § 2]{Kato}) with respect to $\alpha\in[0,\pi)$. Indeed, the family $-\partial_x^2+e^{i\alpha}x^2$ is analytic of type (A) (its domain is $B^2(\R)$ and thus it does not depend on $\alpha\in[0,\pi)$) and the eigenvalues (which are explicit!) continuously move with respect to $\alpha$. The rank of the associated Riesz projection is then constant, equal to $1$ (the rank when $\alpha=0$).

Let us now discuss the resolvent estimate. Since each eigenvalue is simple and
isolated, the resolvent has a simple pole there, which gives the
required estimate~\eqref{eq.resQ0} in a small neighborhood of the
eigenvalue. Since the disc of radius $C$ contains only a finite number
of eigenvalues, the result follows.

\end{proof}
Let us now explain why the spectrum of $Q_0^w$ gives an approximation of the spectrum of $\mathsf{P}_h^{\mathrm{eff}}$ in $D(0,Ch)$.

\subsection{Locating the spectrum of $\Peff$}
The spectrum of $\Peff$ in $D(0,Ch)$ is close to the one of $Q_0^w$. Let us consider $\lambda\in\mathrm{sp}(\Peff)\cap D(0,Ch)$ and $\psi$ be a corresponding eigenfunction. We have
\[\Peff\psi=\lambda\psi\,.\]
Similarly to Lemma \ref{lem.microlocalisation}, one can check that $\psi$ is microlocalized near $X=0$ at a scale $h^{\delta}$ with $\delta\in\left(0,\frac12\right)$. We infer that
\[\|(Q_0^w-\lambda)\psi\|\leq \tilde Ch^{3\delta}\|\psi\|\,.\] 
With \eqref{eq.resQ0}, we get
\[\mathsf{dist}(\lambda,\mathrm{sp}(Q_0^w))\leq \check Ch^{3\delta}\,.\]
Therefore, when $h$ is small enough, the spectrum of $\Peff$ is close to the one of $Q_0^w$ at a distance bounded by $h^{\frac32-\kappa}$ for all $\kappa>0$. 
\subsection{Comparison of the spectral projections and resolvent bound}
Let us explain why there is exactly one simple eigenvalue of $\Peff$ in each disc $D(\mu,h^{\frac32-\kappa})$ with $\mu\in D(0,Ch)\cap\mathrm{sp}(Q_0^w)$.

It is enough to prove that the
Riesz projections associated with $\Peff$ and $Q_0^w$ have the same rank. More precisely, we let
\[ \Pi^{\mathrm{eff}}_h =-\frac{1}{2 i \pi} \int_{\mathscr{C}_{h}} ( \Peff - z)^{-1} \dd z \,, \quad \Pi^0_h =- \frac{1}{2i \pi} \int_{\mathscr{C}_h} ( Q_0^w - z)^{-1} \dd z \,, \]
where $\mathscr{C}_h$ is the circle of center $\mu$ and radius $h^{\frac32-\kappa}$. We can estimate the norm $\Vert \Pi^{\mathrm{eff}}_h - \Pi^0_h \Vert$ by using the same method as in Proposition \ref{prop.Lh.to.Lh0}. 

Let $\chi_1, \chi_2 \in \Cinf(\R^2)$ be cutoff functions such that $\chi_1 + \chi_2 = 1$ and $\chi_2$ supported in $\{|X|\leq h^{\delta}\}$. Then, we use the resolvent formula to get
\[ \chi_2^w (\Pi^{\mathrm{eff}}_h - \Pi^0_h) = \frac{1}{2i\pi} \int_{\mathscr{C}_h} \underbrace{\chi_2^w (Q_0^w - z)^{-1}R_h^w}_{\mathscr{O}(h^{-\frac32+\kappa}(h^{3\delta}+h^{1+\delta}))} (\Peff - z)^{-1} \dd z \,, \]
where $R_h$ is defined in \eqref{eq.mueffexpansion}. By choosing $\delta$ close enough to $\frac12$, we get, for some $\alpha>0$,
\[ \Vert \chi_2^w (\Pi^{\mathrm{eff}}_h - \Pi^0_h ) \Vert \leq C h^{\alpha} \int_{\mathscr{C}_h} \Vert (\Peff -z)^{-1} \Vert \dd z \,. \]
Since $\chi_1$ is supported in $\{|X|\geq ch^{\delta}\}$, we have, for all $z\in D(0,Ch)$,
\begin{equation}\label{eq.chi1res}
\Vert \chi_1^w (\Peff - z)^{-1} \Vert\leq C h^{-2 \delta} \,,\quad \Vert \chi_1^w (Q_0^w - z)^{-1} \Vert\leq C h^{-2 \delta}\,.
\end{equation}
We deduce that
\[ \Vert \chi_1^w (\Pi^{\mathrm{eff}}_h - \Pi^0_h ) \Vert \leq C h^{\frac32-\kappa} h^{-2 \delta} \,. \]
Summing up the $\chi_2$ and the $\chi_1$ parts, we deduce that
\begin{equation}\label{eq.006}
\Vert \Pi^{\mathrm{eff}}_h - \Pi^0_h \Vert \leq C h^{\frac32-\kappa - 2 \delta} + C h^{\alpha} \int_{\mathscr{C}_h} \Vert (\Peff -z)^{-1} \Vert \dd z \,. 
\end{equation}
We must estimate the resolvent appearing in the right-hand-side. 

For all $z\in D(0,Ch)$ such that $\mathsf{dist}(z,\mathrm{sp}(Q_0^w))\geq h^{\frac32-\kappa}$,
\[ \Vert \chi_2^w \left( (\Peff -z)^{-1} - (Q_0^w -z)^{-1} \right) \Vert \leq C h^{-\frac32+\kappa}(h^{3\delta}+h^{1+\delta})\|(\Peff-z)^{-1}\| \,.\]
Thus, with \eqref{eq.chi1res},
\[ \Vert (\Peff -z)^{-1} \Vert \leq \Vert (Q_0^w-z)^{-1} \Vert + C h^{-2 \delta} + C \underbrace{h^{-\frac32+\kappa}(h^{3\delta}+h^{1-\delta})}_{=o(1)}  \Vert (\Peff -z)^{-1} \Vert \,, \]
which yields
\[ \Vert (\Peff -z)^{-1} \Vert \leq C\Vert (Q_0^w-z)^{-1} \Vert + C h^{-2 \delta}\,.\]
With \eqref{eq.resQ0}, we deduce that, for all $z\in D(0,Ch)$ such that $\mathsf{dist}(z,\mathrm{sp}(Q_0^w))\geq h^{\frac32-\kappa}$,
\begin{equation}\label{eq.resboundPeff}
\begin{split} 
\Vert (\Peff -z)^{-1} \Vert \leq \frac{C}{\mathsf{dist}(z,\mathrm{sp}(Q_0^w))} + C h^{-2 \delta}&=\frac{C+C h^{-2 \delta}\mathsf{dist}(z,\mathrm{sp}(Q_0^w))}{\mathsf{dist}(z,\mathrm{sp}(Q_0^w))}\\
&\leq \frac{\tilde C}{\mathsf{dist}(z,\mathrm{sp}(Q_0^w))}\leq\tilde C h^{-\frac32+\kappa}\,.
\end{split}
\end{equation}
With \eqref{eq.006}, this provides us with
\begin{equation*}
\Vert \Pi^{\mathrm{eff}}_h - \Pi^0_h \Vert \leq C h^{\frac32-\kappa - 2 \delta} + 2\pi C\tilde C h^{\alpha}=o(1)\,. 
\end{equation*}
Therefore, for $h$ small enough, $\Vert \Pi^{\mathrm{eff}}_h - \Pi^0_h \Vert<1$, and the spectral projections have the same rank (that is rank one). We deduce that there is exactly one eigenvalue of $\Peff$ at a distance of $h^{\frac32-\kappa}$ near the spectrum of $Q_0^w$ in the disc $D(0,Ch)$. With \eqref{eq.resboundPeff}, this implies that, for all $z\in D(0,Ch)$ such that $\mathsf{dist}(z,\mathrm{sp}(Q_0^w))\geq h^{\frac32-\kappa}$, 
\[\Vert (\Peff -z)^{-1} \Vert \leq \frac{C}{\mathsf{dist}(z,\mathrm{sp}(\Peff))}\,.\]

\bibliographystyle{abbrv}
\bibliography{biblio-MRV.bib}

\begin{thebibliography}{10}

\bibitem{ALP-R2}
Y.~Almog, B.~Helffer, and X.-B. Pan.
\newblock Superconductivity near the normal state under the action of electric
  currents and induced magnetic fields in {$\Bbb R^2$}.
\newblock {\em Comm. Math. Phys.}, 300(1):147--184, 2010.

\bibitem{BHR21}
V.~Bonnaillie-No\"{e}l, F.~H\'{e}rau, and N.~Raymond.
\newblock Pure magnetic tunnelling effect in two dimensions.
\newblock {\em To appear in Inventiones Mathematicae}, 2021.

\bibitem{Bony99}
J.-M. Bony.
\newblock Sur l'in\'{e}galit\'{e} de {F}efferman-{P}hong.
\newblock In {\em Seminaire: \'{E}quations aux {D}\'{e}riv\'{e}es {P}artielles,
  1998--1999}, S\'{e}min. \'{E}qu. D\'{e}riv. Partielles, pages Exp. No. III,
  16. \'{E}cole Polytech., Palaiseau, 1999.

\bibitem{Boulton02}
L.~S. Boulton.
\newblock Non-self-adjoint harmonic oscillator, compact semigroups and
  pseudospectra.
\newblock {\em J. Operator Theory}, 47(2):413--429, 2002.

\bibitem{CFK20}
L.~Cossetti, L.~Fanelli, and D.~Krej\v{c}i\v{r}\'{\i}k.
\newblock Absence of eigenvalues of {D}irac and {P}auli {H}amiltonians via the
  method of multipliers.
\newblock {\em Comm. Math. Phys.}, 379(2):633--691, 2020.

\bibitem{Davies99}
E.~B. Davies.
\newblock Pseudo-spectra, the harmonic oscillator and complex resonances.
\newblock {\em R. Soc. Lond. Proc. Ser. A Math. Phys. Eng. Sci.},
  455(1982):585--599, 1999.

\bibitem{FKV18}
L.~Fanelli, D.~Krej\v{c}i\v{r}\'{\i}k, and L.~Vega.
\newblock Spectral stability of {S}chr\"{o}dinger operators with subordinated
  complex potentials.
\newblock {\em J. Spectr. Theory}, 8(2):575--604, 2018.

\bibitem{FournaisHelffer}
S.~Fournais and B.~Helffer.
\newblock {\em Spectral methods in surface superconductivity}.
\newblock Number~77 in Progress in Nonlinear Differential Equations and their
  Applications. Birkh\"{a}user Boston Inc., 2010.

\bibitem{GBRVN21}
Y.~Guedes~Bonthonneau, N.~Raymond, and S.~V\~{u}~Ng\d{o}c.
\newblock Exponential localization in 2{D} pure magnetic wells.
\newblock {\em Ark. Mat.}, 59(1):53--85, 2021.

\bibitem{Helffer13}
B.~Helffer.
\newblock {\em Spectral theory and its applications}, volume 139 of {\em
  Cambridge Studies in Advanced Mathematics}.
\newblock Cambridge University Press, Cambridge, 2013.

\bibitem{HelKor14-2}
B.~Helffer and Y.~Kordyukov.
\newblock Accurate semiclassical spectral asymptotics for a two-dimensional
  magnetic schr\"{o}dinger operator.
\newblock {\em Annales Henri Poincar\'{e}}, 16:1651--1688, 2014.

\bibitem{Hitrik04}
M.~Hitrik.
\newblock Boundary spectral behaviour for semiclassical operators in one
  dimension.
\newblock {\em International Mathematics Research Notices}, 64:3417--3438,
  2004.

\bibitem{HormanderIII}
L.~H\"{o}rmander.
\newblock {\em The analysis of linear partial differential operators. {III}}.
\newblock Classics in Mathematics. Springer, Berlin, 2007.
\newblock Pseudo-differential operators, Reprint of the 1994 edition.

\bibitem{Kato}
T.~Kato.
\newblock {\em Perturbation theory for linear operators}.
\newblock Classics in Mathematics. Springer-Verlag, Berlin, 1995.
\newblock Reprint of the 1980 edition.

\bibitem{Keraval}
P.~Keraval.
\newblock {\em Formules de Weyl par réduction de dimension. Application à des
  Laplaciens électro-magnétiques.}
\newblock PhD thesis, Université de Rennes 1, 2018.

\bibitem{Martinez07}
A.~Martinez.
\newblock A general effective {H}amiltonian method.
\newblock {\em Atti Accad. Naz. Lincei Rend. Lincei Mat. Appl.},
  18(3):269--277, 2007.

\bibitem{PS06}
K.~Pravda-Starov.
\newblock A complete study of the pseudo-spectrum for the rotated harmonic
  oscillator.
\newblock {\em J. London Math. Soc. (2)}, 73(3):745--761, 2006.

\bibitem{Raymond}
N.~Raymond.
\newblock {\em Bound States of the Magnetic Schr\"{o}dinger Operator}.
\newblock Number~27 in Tracts in Mathematics. European Mathematical Society,
  2017.

\bibitem{Birkhoff2D}
N.~Raymond and S.~V\~{u}~Ngoc.
\newblock {Geometry and Spectrum in 2D Magnetic wells}.
\newblock {\em Annales de l'Institut Fourier}, 65(1):137--169, 2015.

\bibitem{rouby-17}
O.~Rouby.
\newblock {Bohr–Sommerfeld Quantization Conditions for Non-Selfadjoint
  Perturbations of Selfadjoint Operators in Dimension One}.
\newblock {\em Int. Math. Res. Not.}, 2018(7):2156--2207, 01 2017.

\bibitem{Sambou}
D.~Sambou.
\newblock A simple criterion for the existence of nonreal eigenvalues for a
  class of 2{D} and 3{D} {P}auli operators.
\newblock {\em Linear Algebra Appl.}, 529:51--88, 2017.

\bibitem{Sambou2}
D.~Sambou.
\newblock Spectral non-self-adjoint analysis of complex {D}irac, {P}auli and
  {S}chr\"{o}dinger operators with constant magnetic fields of full rank.
\newblock {\em Asymptot. Anal.}, 111(2):113--136, 2019.

\bibitem{Sj92}
J.~Sj\"{o}strand.
\newblock Semi-excited states in nondegenerate potential wells.
\newblock {\em Asymptotic Anal.}, 6(1):29--43, 1992.

\bibitem{SZ07}
J.~Sj\"{o}strand and M.~Zworski.
\newblock Elementary linear algebra for advanced spectral problems.
\newblock volume~57, pages 2095--2141. 2007.
\newblock Festival Yves Colin de Verdi\`ere.

\bibitem{Yoshida}
N.~Yoshida.
\newblock {Eigenvalues and eigenfunctions for the two dimensional Schrödinger
  operator with strong magnetic field}.
\newblock {\em Asymptotic Analysis}, 120(1-2):175--197, 2020.

\bibitem{Zworski}
M.~Zworski.
\newblock {\em Semiclassical Analysis}.
\newblock Number 138 in Graduate Studies in Mathematics. American Mathematical
  Society, 2012.

\end{thebibliography}

\end{document}